\documentclass[11pt,letterpaper]{article}

\usepackage[utf8]{inputenc}
\usepackage{comment}
\usepackage{microtype}
\usepackage{graphicx}
\usepackage{booktabs} 
\usepackage{scalerel}
\usepackage[margin=1in]{geometry}
\usepackage[style=alphabetic, maxbibnames=15, maxcitenames=15, natbib=true, maxalphanames=5, backend=biber, sorting=nty, backref=true]{biblatex}
\DefineBibliographyStrings{english}{backrefpage = {page},backrefpages = {pages},}

\usepackage{amssymb}
\usepackage{amsthm}
\usepackage{amsmath}
\usepackage{mathtools}

\usepackage{nicefrac}

\usepackage{algorithmic}
\usepackage{algorithm}
\usepackage[shortlabels]{enumitem}

\usepackage{url}
\usepackage{float}

\usepackage{pifont}
\usepackage{hhline}
\usepackage{multirow}

\usepackage{subfig}
\usepackage{graphicx}
\usepackage{amsmath}
\usepackage{amsthm}
\usepackage{amssymb}
\usepackage{cancel}

\usepackage{multirow,multicol}
\usepackage{pifont}

\def\grad{\nabla}

\def\bb{\mathbf{b}}

\def\bu{\mathbf{u}}

\def\bx{\mathbf{x}}  
\def\by{\mathbf{y}}
\def\bz{\mathbf{z}}
\def\bA{\mathbf{A}}

\def\bI{\mathbf{I}}

\def\bQ{\mathbf{Q}}

\def\cX{\mathcal{X}}

\def\cZ{\mathcal{Z}}

\def\smskip{\smallskip}

\def\texitem#1{\par\smskip\noindent\hangindent 25pt
               \hbox to 25pt {\hss #1 ~}\ignorespaces}

\def\fprod#1{\left\langle#1\right\rangle}

\newcommand{\BEAS}{\begin{eqnarray*}}
\newcommand{\EEAS}{\end{eqnarray*}}
\newcommand{\BEA}{\begin{eqnarray}}
\newcommand{\EEA}{\end{eqnarray}}
\newcommand{\BEQ}{\begin{eqnarray}}
\newcommand{\EEQ}{\end{eqnarray}}
\newcommand{\BIT}{\begin{itemize}}
\newcommand{\EIT}{\end{itemize}}
\newcommand{\BNUM}{\begin{enumerate}}
\newcommand{\ENUM}{\end{enumerate}}

\newcommand{\BA}{\begin{array}}
\newcommand{\EA}{\end{array}}

\newcommand{\reals}{\mathbb{R}}

\DeclareMathOperator*{\argmin}{\arg\!\min}

\newif\ifpagenumbering
\pagenumberingtrue

\pagenumberingfalse

\newsavebox{\theorembox}
\newsavebox{\lemmabox}
\newsavebox{\defnbox}
\newsavebox{\corollarybox}
\newsavebox{\remarkbox}
\newsavebox{\assbox}
\savebox{\theorembox}{\noindent\bf Theorem}
\savebox{\lemmabox}{\noindent\bf Lemma}
\savebox{\defnbox}{\noindent\bf Definition}
\savebox{\corollarybox}{\noindent\bf Corollary}
\savebox{\remarkbox}{\noindent\bf Remark}

\newtheorem{assumption}{Assumption}
\newtheorem{definition}{Definition}
\newtheorem{theorem}{Theorem}
\newtheorem{lemma}[theorem]{Lemma}
\newtheorem{proposition}[theorem]{Proposition}

\newtheorem{remark}{Remark}
\numberwithin{assumption}{section}
\numberwithin{definition}{section}
\numberwithin{theorem}{section}
\numberwithin{remark}{section}

\usepackage{hyperref}
\hypersetup{colorlinks,citecolor=blue,linktocpage,breaklinks=true}
\begin{document}

\title{An Accelerated Gradient Method for Convex Smooth \\Simple Bilevel Optimization}

\author{Jincheng Cao\thanks{Department of Electrical and Computer Engineering, The University of Texas at Austin, Austin, TX, USA \qquad\{jinchengcao@utexas.edu, rjiang@utexas.edu, mokhtari@austin.utexas.edu\}}         \quad
        Ruichen Jiang$^*$ 
   \quad
   Erfan Yazdandoost Hamedani\thanks{Department of Systems and Industrial Engineering, The University of Arizona, Tucson, AZ, USA \qquad\{erfany@arizona.edu\}}
   \quad
        Aryan Mokhtari$^*$
}


\date{}

\maketitle
\begin{abstract} 
In this paper, we focus on simple bilevel optimization problems, where we minimize a convex smooth objective function over the optimal solution set of another convex smooth constrained optimization problem.
We present a novel bilevel optimization method that locally approximates the solution set of the lower-level problem using a cutting plane approach and employs an accelerated gradient-based update to reduce the upper-level objective function over the approximated solution set.
We measure the performance of our method in terms of suboptimality and infeasibility errors and provide non-asymptotic convergence guarantees for both error criteria. Specifically, when the feasible set is compact, we show that our method requires at most $\mathcal{O}(\max\{1/\sqrt{\epsilon_{f}}, 1/\epsilon_g\})$ iterations to find a solution that is $\epsilon_f$-suboptimal and $\epsilon_g$-infeasible. 
Moreover, under the additional assumption that the lower-level objective satisfies the $r$-th Hölderian error bound,  we show that our method achieves an iteration complexity of $\mathcal{O}(\max\{\epsilon_{f}^{-\frac{2r-1}{2r}},\epsilon_{g}^{-\frac{2r-1}{2r}}\})$, which matches the optimal complexity of single-level convex constrained optimization when $r=1$. 
\end{abstract}

\newpage
\newpage

\section{Introduction}
In this paper, we investigate a class of bilevel optimization problems known as simple bilevel optimization, in which we aim to minimize an upper-level objective function over the solution set of a corresponding lower-level problem. Recently this class of problems has gained great attention due to their board applications in continual learning \cite{borsos2020coresets}, hyper-parameter optimization \cite{franceschi2018bilevel,shaban2019truncated}, meta-learning \cite{rajeswaran2019meta,bertinetto2018meta}, and over-parameterized machine learning \cite{jiangconditional, cao2024projection, samadi2023achieving}. Specifically, we focus on the following bilevel optimization problem:
\begin{equation}\label{eq:bi-simp}
    \min_{\bx\in \reals^n}~f(\bx)\qquad \hbox{s.t.}\quad  \bx\in\argmin_{\bz\in \cZ}~g(\bz),
\end{equation} 
where $\cZ$ is a convex set and $f,g:\reals^n\to \reals$ are convex and continuously differentiable functions on an open set containing $\cZ$. We assume that the lower-level objective function $g$ is convex but may not be strongly convex. Hence, the lower-level problem may have multiple optimal solutions. 
Throughout the paper, we will use $\bx^*$ to denote an optimal solution of problem~\eqref{eq:bi-simp}. Further, we define $f^* \triangleq f(\bx^*)$ and $g^* \triangleq g(\bx^*)$, which represent the optimal value of problem~\eqref{eq:bi-simp} and the optimal value of the lower-level objective $g$, respectively.  
This class of problems is often referred to as the ``simple bilevel problem" in the literature \cite{dempe2010optimality, dutta2020algorithms, shehu2021inertial} to distinguish it from general settings where the lower-level problem is parameterized by some upper-level variables.

The main challenge in solving problem \eqref{eq:bi-simp} arises from the fact that the feasible set, i.e., the optimal solution set of the lower-level problem, lacks a simple characterization and is not explicitly provided. Consequently, direct application of projection-based or projection-free methods is infeasible, as projecting onto or solving a linear minimization problem over such an implicitly defined feasible set is intractable.
Instead, our approach begins by constructing an approximation set that possesses specific properties, serving as a surrogate for the true feasible set of the bilevel problem. In Section~\ref{sec:algorithm}, we delve into the details of how such a set is constructed. By using this technique and building upon the idea of the projected accelerated gradient method, we establish the best-known complexity bounds for solving problem~\eqref{eq:bi-simp}.

To provide a clearer context for our results,  note that the best-known complexity bound for achieving an $\epsilon$-accurate solution in single-level convex constrained optimization problems is $\mathcal{O}(\epsilon^{-0.5})$, as demonstrated in \cite{beck2009fast}. This bound is optimal and was achieved using the accelerated proximal method or FISTA (Fast Iterative Shrinkage-Thresholding Algorithm), which plays a canonical role in the development of our algorithm as well.

While the literature on bilevel optimization is not as extensive as that for single-level optimization, there have been recent non-asymptotic results 
for solving this class of problems, which we summarize in Table~\ref{sum}. 
Specifically, these results aim to establish convergence rates on the \emph{infeasibility gap} $g(\bx_k)-g^*$ and the \emph{suboptimality gap} $f(\bx_k)-f^*$ after $k$ iterations. \citet{kaushik2021method} demonstrated that an iterative regularization-based method achieves a convergence rate of  $\mathcal{O}(1/k^{0.5-b})$ in terms of suboptimality and a rate of $\mathcal{O}(1/k^b)$ in terms of infeasibility, where $b \in (0,0.5)$ is a user-defined parameter. Therefore, if we set $b = 0.25$ to balance these two rates, it would require an iteration complexity of $O(\max\{1/\epsilon_f^4,1/\epsilon_{g}^4\})$ to find a solution that is $\epsilon_{f}$-optimal and
$\epsilon_{g}$-infeasible.
Later, the Bi-Sub-Gradient (Bi-SG) algorithm was proposed by \citet{merchav2023convex}
to address convex simple bilevel optimization problems with nonsmooth upper-level objective functions. They showed convergence rates of $\mathcal{O}(1/k^{1-\alpha})$ and $\mathcal{O}(1/k^{\alpha})$ in terms of
suboptimality and infeasibility,
respectively, where $\alpha \in(0.5,1)$ serves as a hyper-parameter. This implies that if we balance the rates by setting $\alpha=0.5$, the iteration complexity of Bi-SG is $\mathcal{O}(\max\{1/\epsilon_f^{2},1/\epsilon_g^{2}\})$. 
Additionally, \citet{shen2023online} introduced a structure-exploiting method and presented an iteration complexity of  $O(\max\{1/\epsilon_f^2,1/\epsilon_{g}^2\})$ for it, when the upper-level objective is convex and the lower-level objective is convex and smooth. Note that imposing additional assumptions on the upper-level function, such as smoothness or strong convexity, does not result in faster rates for their method.

Recently, \cite{jiangconditional} presented a projection-free conditional gradient method (CG-BiO) that uses a cutting plane to approximate the solution set of the lower-level problem. Assuming both upper- and lower-level objective functions are convex and smooth, CG-BiO achieves a complexity of $\mathcal{O}(\max\{1/\epsilon_f,1/\epsilon_g\})$. 
Since the suboptimality gap $f(\hat{\bx})-f^*$ may be negative for an infeasible point $\hat{\bx}$, a more desirable metric is the \emph{absolute suboptimality gap} $|f(\hat{\bx})-f^*|$. To ensure this, \cite{jiangconditional} introduced the Hölderian error bound condition on $g$. Specifically, under the $r$-th order Hölderian error bound condition, CG-BiO finds a solution $\hat{\bx}$ with $|f(\hat{\bx})-f^*| \leq \epsilon_f$ and $g(\hat{\bx})-g^* \leq \epsilon_g$ after $\mathcal{O}(\max\{1/\epsilon_{f}^r,1/\epsilon_g\})$ iterations.
More recently, \cite{samadi2023achieving} introduced the regularized proximal accelerated method (R-APM), which runs the proximal accelerated gradient method on a weighted sum of the upper- and lower-level objective functions. Assuming both functions are convex and smooth, they established a complexity bound of $\mathcal{O}(\max\{1/\epsilon_f,1/\epsilon_g\})$ to find an $(\epsilon_f,\epsilon_g)$ solution. This bound is worse than the $\mathcal{O}(\max\{1/\sqrt{\epsilon_f},1/\epsilon_{g}\})$ complexity achieved by our proposed AGM-BiO method, assuming the feasible set $\mathcal{Z}$ is compact.
Additionally, \cite{samadi2023achieving} showed that when the lower-level objective function $g$ satisfies the weak sharpness property (equivalent to the Hölderian error bound condition with $r=1$), R-APM finds an $(\epsilon_f,\epsilon_g)$-absolute optimal solution after at most $\mathcal{O}(\max\{1/\sqrt{\epsilon_f},1/\sqrt{\epsilon_g}\})$ iterations. This result is comparable to our convergence result for AGM-BiO, which considers a more general Hölderian error bound condition.

\begin{table*}[!tbp]
    \vspace{-2mm}
    \centering
    \caption{Non-asymptotic results on simple bilevel optimization. (\textcircled{1}: with a first-order H\"olderian error bound assumption on $g$; \textcircled{r}: with an $r$th-order ($r\geq1$) H\"olderian error bound assumption on $g$; \textcircled{A}: with an additional assumption implying that the projection onto the sublevel set of $f$ is easy to compute.)}\label{sum} 
    \resizebox{\textwidth}{!}{%
        \begin{tabular}{cccccc}
            \toprule
\multirow{2}{*}{References}         & Upper level                                            & \multicolumn{2}{c}{Lower level}                                       & \multicolumn{2}{c}{Convergence}                                    \\ \cmidrule(r){2-6}
    & Objective $f$                                                          & Objective $g$            & Feasible set $\mathcal{Z}$                          & Upper level &   Lower level                                                                                                                \\ \midrule
           a-IRG \cite{kaushik2021method}    & Convex, Lipschitz    & Convex, Lipschitz      & Closed    & $\mathcal{O}(1/\epsilon_{f}^{4})$ & $\mathcal{O}(1/\epsilon_{g}^{4})$
           \\ \midrule
            Bi-SG \cite{merchav2023convex} & Convex, Nonsmooth & Convex, Composite & Closed &$\mathcal{O}(1/\epsilon_{f}^{\frac{1}{1-\alpha}})$  &$\mathcal{O}(1/\epsilon_{g}^{\frac{1}{\alpha}}),\alpha \in (0.5,1)$
            \\ \midrule
            SEA \cite{shen2023online} & Convex & Convex, Smooth & Compact &$\mathcal{O}(1/\epsilon_{f}^2)$  &$\mathcal{O}(1/\epsilon_{g}^2)$
            \\ \midrule

           CG-BiO \cite{jiangconditional} & Convex, Smooth & Convex, Smooth & Compact &$\mathcal{O}(1/\epsilon_{f})$  &$\mathcal{O}(1/\epsilon_{g})$
           \\ \midrule
           R-APM \cite{samadi2023achieving} & Convex, Smooth & Convex, Composite & Closed &$\mathcal{O}(1/\epsilon_{f})$  &$\mathcal{O}(1/\epsilon_{g})$
           \\ \midrule
           \textbf{AGM-BiO (Ours)}                                        & Convex, Smooth                                                   & Convex, Smooth                      & Compact       &$\mathcal{O}(1/\epsilon_{f}^{0.5})$  &$\mathcal{O}(1/\epsilon_{g})$
           \\ \midrule
           R-APM \textcircled{1} \cite{samadi2023achieving} & Convex, Smooth & Convex, Composite & Closed &$\mathcal{O}(1/\epsilon_{f}^{0.5})$  &$\mathcal{O}(1/\epsilon_{g}^{0.5})$
           \\ \midrule
           Bisec-BiO \textcircled{A} \cite{wang2024near} & Convex, Composite & Convex, Composite & Closed 
            & \multicolumn{2}{c} {$\tilde{\mathcal{O}}(\max\{1/\epsilon_{f}^{0.5}, 1/\epsilon_{g}^{0.5}\})$}
           
           \\ \midrule
           \textbf{AGM-BiO \textcircled{1} (Ours)}                                        & Convex, Smooth                                                   & Convex, Smooth                      & Closed       &$\mathcal{O}(1/\epsilon_{f}^{0.5})$  &$\tilde{\mathcal{O}}(1/\epsilon_{g}^{0.5})$
            \\ \midrule
           PB-APG \textcircled{r} \cite{chen2024penalty} & Convex, Composite & Convex, Composite & Compact 
           & \multicolumn{2}{c} {$\mathcal{O}(1/\epsilon_f^{0.5r})+\mathcal{O}(1/\epsilon_g^{0.5} )$}

           \\ \midrule
           \textbf{AGM-BiO \textcircled{r} (Ours)}                                        & Convex, Smooth                                                   & Convex, Smooth                      & Closed       &$\tilde{\mathcal{O}}(1/\epsilon_{f}^{\frac{2r-1}{2r}})$  &$\tilde{\mathcal{O}}(1/\epsilon_{g}^{\frac{2r-1}{2r}})$
           \\ \bottomrule
        \end{tabular}%
     }
     \vspace{-1mm}
\end{table*}

\textbf{Contribution.} In this paper, we present a novel accelerated gradient-based bilevel optimization method with state-of-the-art non-asymptotic guarantees in terms of both suboptimality and infeasibility.
At each iteration, our proposed AGM-BiO method uses a cutting plane to linearly approximate the solution set of the lower-level problem, and then runs a variant of the projected accelerated gradient update on the upper-level objective function. Next, we summarize our theoretical guarantees for the AGM-BiO method:
\begin{itemize}
\vspace{-2mm}
    \item When the feasible set $\mathcal{Z}$ is compact, we show that AGM-BiO finds $\hat{\bx}$ that satisfies $f(\hat{\bx}) - f^* \leq \epsilon_{f}$ and $g(\hat{\bx}) - g^* \leq \epsilon_{g}$ within $\mathcal{O}(\max\{1/\sqrt{\epsilon_f},1/\epsilon_{g}\})$ iterations, where $f^*$ is the optimal value of problem~\eqref{eq:bi-simp} and $g^*$ is the optimal value of the lower-level problem. 
    \vspace{-2mm}
    \item With an additional $r$-th-order ($r\geq1$) Hölderian error bound assumption on the lower-level problem, AGM-BiO finds $\hat{\bx}$ satisfying $f(\hat{\bx})-f^* \leq \epsilon_{f}$ and $g(\hat{\bx})-g^* \leq \epsilon_{g}$ within $\mathcal{O}(\max\{\epsilon_{f}^{-\frac{2r-1}{2r}},\epsilon_{g}^{-\frac{2r-1}{2r}}\})$ iterations. Moreover, it achieves the stronger guarantee that $|f(\hat{\bx})-f^*| \leq \epsilon_f$ and $g(\hat{\bx})-g^* \leq \epsilon_{g}$ within $\mathcal{O}(\max\{\epsilon_{f}^{-\frac{2r-1}{2}},\epsilon_{g}^{-\frac{2r-1}{2r}}\})$ iterations. 
\end{itemize}

\vspace{-1mm}
These bounds all achieve the best-known complexity bounds in terms of both suboptimality and infeasibility guarantees for the considered settings. 
All the non-asymptotic results are summarized and compared in Table~\ref{sum}.

\noindent \textbf{Discussions on two concurrent works}. The authors in \cite{wang2024near} proposed a bisection algorithm with a total operation complexity of $\mathcal{\Tilde{O}}(\max\{\epsilon_{f}^{-0.5}, \epsilon_{g}^{-0.5}\})$ to find an $(\epsilon_{f}, \epsilon_{g})$-optimal solution, assuming the upper-level objective $f$ meets specific criteria. Specifically, Assumption 1(iv) in \cite{wang2024near} implies the ability to compute the projection onto the sublevel set of the upper-level function $f$. However, this assumption may not hold for general functions, such as the mean squared loss function in our over-parameterized regression example in Section ~\ref{sec:experiment}.
In \cite{chen2024penalty}, the authors introduced the penalty-based accelerated proximal gradient method (PB-APG) for solving simple bilevel optimization problems with the $r$-th order Hölderian error bound assumption on the lower-level objective $g$. Their algorithm, similar to \cite{samadi2023achieving}, runs the accelerated proximal gradient method on a weighted sum of the upper and lower-level objective functions. PB-APG achieves a complexity of $\mathcal{O}(\epsilon_f^{-0.5r})+\mathcal{O}(\epsilon_g^{-0.5})$ to find an $(\epsilon_{f}, \epsilon_{g})$-optimal solution. The term $\mathcal{O}(1/\epsilon_{f}^{0.5r})$ can become significantly large as the order of the Hölderian error bound $r$ increases.
In contrast, our algorithm, AGM-BiO, avoids this issue, requiring at most $\tilde{\mathcal{O}}(\max\{\epsilon_{f}^{-\frac{2r-1}{2r}},\epsilon_{g}^{-\frac{2r-1}{2r}}\})$ iterations to achieve an $(\epsilon_{f}, \epsilon_{g})$-optimal solution. Therefore, regardless of how large $r$ is, the worst-case complexity for AGM-BiO is $\tilde{\mathcal{O}}(\max\{\epsilon_{f}^{-1},\epsilon_{g}^{-1}\})$. Thus, our method achieves a better rate than PB-APG when $r > 1$.


\noindent\textbf{Additional related work.}
Previous work has explored ``asymptotic" results for simple bilevel problems, dating back to Tikhonov-type regularization introduced in \cite{tikhonov1977solutions}. In this approach, the objectives of both levels are combined into a single-level problem using a regularization parameter $\sigma > 0$ and as $\sigma \to 0$ the solutions of the regularized single-level problem approaches a solution to the bilevel problem in \eqref{eq:bi-simp}. 
Further, \citet{solodov2007explicit} proposed the explicit descent method that solves problem \eqref{eq:bi-simp} when upper and lower-level functions are smooth and convex. This result was further extended to a non-smooth setting in \cite{solodov2007bundle}. The results in both \cite{solodov2007explicit} and \cite{solodov2007bundle} only indicated that both upper and lower-level objective functions converge asymptotically. Moreover, \citet{helou2017} proposed the $\epsilon$-subgradient method to solve simple bilevel problems and showed its asymptotic convergence. Specifically, they assumed the upper-level objective function to be convex and utilized two different algorithms, namely, the Fast Iterative Bilevel Algorithm (FIBA) and Incremental Iterative Bilevel Algorithm (IIBA), that consider smooth and non-smooth lower-level objective functions, respectively. 

Some studies have only established non-asymptotic convergence rates for the lower-level problem. One of the pioneering methods in this category is the minimal norm gradient (MNG) method, introduced by \citet{beck2014first}. This method assumes that the upper-level objective function is smooth and strongly convex, while the lower-level objective function is smooth and convex. The authors showed that the lower-level objective function reaches an iteration complexity of $\mathcal{O}(1/\epsilon^{2})$.
Subsequently, the Bilevel Gradient SAM (BiS-SAM) method was introduced by \citet{sabach2017first}, and it was proven to achieve a complexity of $\mathcal{O}(1/\epsilon)$ for the lower-level problem. A similar rate of convergence was also attained in \cite{malitsky2017chambolle}.

\section{Preliminaries}
In this section, we state the necessary assumptions and introduce the notions of optimality that we use in the paper.
\subsection{Assumptions and Definitions}
We focus on the case where both the upper and lower-level functions $f$ and $g$ are convex and smooth. Formally, we make the following assumptions.

\begin{assumption}\label{ass:1}
    Let $\|\cdot\|$ be an arbitrary norm on $\mathbb{R}^d$ and $\|\cdot\|_*$ be its dual norm. We assume these conditions hold: 
    \begin{enumerate}[label=(\roman*)]
        \item $\mathcal{Z} \subset \mathbb{R}^d$ is convex and compact with diameter $D$, i.e., $\|\mathbf{x}-\mathbf{y}\| \leq D$ for all $\mathbf{x}, \mathbf{y} \in \mathcal{Z}$.
        \item $g$ is convex and continuously differentiable on an open set containing $\mathcal{Z}$, and its gradient is $L_g$-Lipschitz, i.e., $\|\nabla g(\mathbf{x})-\nabla g(\mathbf{y})\|_* \leq L_g\|\mathbf{x}-\mathbf{y}\|$ for all $\mathbf{x}, \mathbf{y} \in \mathcal{Z}$.
        \item $f$ is convex and continuously differentiable and its gradient is Lipschitz with constant $L_f$.
    \end{enumerate}

\end{assumption}

In this paper, we denote the optimal value and the optimal solution set of the lower-level problem as  $g^* \triangleq \min _{\mathbf{z} \in \mathcal{Z}} g(\mathbf{z})$ and $\mathcal{X}_g^* \triangleq$ $\operatorname{argmin}_{\mathbf{z} \in \mathcal{Z}} g(\mathbf{z})$, respectively. By Assumption~\ref{ass:1}, the set $\cX_g^*$ is nonempty, compact, and convex, but in general, not a singleton since $g$ could have multiple optimal solutions on $\mathcal{Z}$, as $g$ is only convex but not strongly convex. Moreover, we use $f^*$ to denote the optimal value and $\bx^*$ to denote an optimal solution of problem~\eqref{eq:bi-simp}. 

Note that in the simple bilevel problem, the suboptimality of a solution $\hat{\bx}$ is measured by the difference $f(\hat{\bx}) - f^*$. Similarly, its infeasibility is indicated by $g(\hat{\bx}) - g^*$. Since it is necessary to control both of these error terms to ensure minimal suboptimality and infeasibility, we formally define an $(\epsilon_f, \epsilon_g)$-optimal solution as follows.

\begin{definition}
    $\left(\left(\epsilon_f, \epsilon_g\right)\right.$-optimal solution). A point $\hat{\mathbf{x}} \in \mathcal{Z}$ is $\left(\epsilon_f, \epsilon_g\right)$-optimal for problem~\eqref{eq:bi-simp} if
$
f(\hat{\mathbf{x}})-f^* \leq \epsilon_f$ and $ g(\hat{\mathbf{x}})-g^* \leq \epsilon_g $.
\end{definition}

This definition is commonly used in simple bilevel optimization literature \cite{jiangconditional,merchav2023convex,cao2024projection,samadi2023achieving}.
Due to the special structure of bilevel optimization, it is not guaranteed that $f(\hat{\bx})-f^*$ will always be positive. To address this, we propose the following definition, utilizing $|f(\hat{\bx})-f^*|$  as the absolute optimal criterion.

\begin{definition}
    $\left(\left(\epsilon_f, \epsilon_g\right)\right.$-absolute optimal solution). A point $\hat{\mathbf{x}} \in \mathcal{Z}$ is $\left(\epsilon_f, \epsilon_g\right)$-absolute optimal for problem~\eqref{eq:bi-simp} if
$
|f(\hat{\mathbf{x}})-f^*| \leq \epsilon_f$ and $|g(\hat{\mathbf{x}})-g^*| \leq \epsilon_g$.
\end{definition}

\section{Algorithm}\label{sec:algorithm}

Before presenting our method, we first introduce a conceptual accelerated gradient method for solving the simple bilevel problem in \eqref{eq:bi-simp}. The first step is to recast it as a constrained optimization problem:  
\begin{equation}\label{eq:bi-simp-recasted}
    \min_{\bx\in \reals^n}~f(\bx)\quad \hbox{s.t.}\quad  \bx\in \mathcal{X}^*_g,
\end{equation} 
where $\mathcal{X}^*_g \triangleq \argmin_{\bz\in \cZ}~g(\bz)$ denotes the solution set of the lower-level objective. Thus, conceptually, we can apply Nesterov's accelerated gradient method (AGM) to obtain a rate of $O(1/k^2)$ on the upper-level objective~$f$. Several variants of AGM have been proposed in the literature; see, e.g., \cite{d2021acceleration}. 
In this paper, we consider a variant proposed in \cite{Tseng2008}. It involves three intertwined sequences of iterates $\{\bx_k\}_{k\geq 0},\{\by_k\}_{k\geq 0},\{\bz_k\}_{k\geq 0}$ and the scalar variables $\{a_k\}_{k \geq 0}$ and $\{A_k\}_{k \geq 0}$. In the first step, we compute the auxiliary iterate $\by_k$ by 
\begin{equation}
    \by_k = \frac{A_k}{A_k+a_k} \bx_k + \frac{a_k}{A_k+a_k} \bz_k. 
\end{equation}
Then in the second step, we update $\bz_{k+1}$ by 
\begin{equation}
    \bz_{k+1} = \Pi_{\mathcal{X}^*_g} (\bz_k-a_k \nabla f(\by_k)),  \label{eq:update_of_z}
\end{equation}
where $\Pi_{\mathcal{X}_g^*}(\cdot)$ denotes the Euclidean projection onto the set $\mathcal{X}_g^*$. Finally, in the third step, we compute $\bx_{k+1}$ and $A_{k+1}$
\begin{equation}
    \bx_{k+1} = \frac{A_k}{A_k+a_k} \bx_k+\frac{a_k}{A_k+a_k} \bz_{k+1}
\end{equation}
and 
\begin{equation}
    A_{k+1} = A_k+a_k
\end{equation}
It can be shown that if the stepsize is selected as $a_k = \frac{k+1}{4 L_f}$, then the suboptimality gap $f(\bx_{k})-f(\bx^*)$ 
 of the iterates generated by the method above converges to zero at the optimal rate of $\mathcal{O}(1/k^2)$. In this case, indeed all the iterates are feasible as it is possible to project onto the set $\mathcal{X}_g^*$.
However, the conceptual method above is not directly implementable for the simple bilevel problem considered in this paper, as the constraint set $\mathcal{X}_g^*$ is not explicitly given. As a result projection onto the set $\mathcal{X}_g^*$ is not computationally tractable.

To address this issue, our key idea involves replacing the implicit set $\mathcal{X}_g^*$ in \eqref{eq:update_of_z} with another set, $\mathcal{X}_k$, which can be explicitly characterized. This approach makes it feasible to perform the Euclidean projection onto $\mathcal{X}_k$. Additionally, $\mathcal{X}_k$ must possess another important property: it should encompass the optimal solution set $\mathcal{X}_g^*$.


\begin{algorithm}[t!]
	\caption{Accelerated Gradient Method for Bilevel Optimization (AGM-BiO)}\label{alg:AGM-BiO}
	\begin{algorithmic}[1]
	\STATE \textbf{Input}: A sequence $\{g_k\}_{k=0}^K$, a scalar $\gamma \in (0,1]$
	\STATE \textbf{Initialization}: $A_0 = 0$, $\bx_0=\bz_0\in \reals^n$
	\FOR{$k = 0,\dots,K$}
	    \STATE Set  $\quad \displaystyle{a_{k} = \gamma\frac{k+1}{4L_{f}}}$
     \STATE Compute 
     $\quad \displaystyle{\by_k = \frac{A_k}{A_k+a_k} \bx_k + \frac{a_k}{A_k+a_k} \bz_k}$

     \vspace{1mm}
     \STATE Compute 
     $\quad \displaystyle{
	        \bz_{k+1} = \Pi_{\mathcal{X}_k} (\bz_k-a_k \nabla f(\by_k))}$, 
         where $$ \displaystyle{ \mathcal{X}_k \triangleq \{\bz\in \mathcal{Z}: g(\by_k)+\fprod{\nabla g(\by_k),\bz-\by_k}\leq g_k\}}$$

      \vspace{-2mm}
     \STATE Compute    $\quad \displaystyle{ \bx_{k+1} = \frac{A_k}{A_k+a_k} \bx_k+\frac{a_k}{A_k+a_k} \bz_{k+1}}$ 
  \vspace{1mm}
		\STATE Update $A_{k+1} = A_k + a_k$
	\ENDFOR
 
    \STATE \textbf{Return:} $\bx_{K}$ 
	\end{algorithmic}
\end{algorithm}

Inspired by the cutting plane approach in \cite{jiangconditional}, we let the set $\cX_k$ be the intersection of $\cZ$ and a halfspace:
\begin{equation}\label{eq:X_k}
    \mathcal{X}_k \triangleq \{\bz\in \mathcal{Z}: g(\by_k)+\fprod{\nabla g(\by_k),\bz-\by_k}\leq g_k\}.
\end{equation}
Here, the auxiliary sequence $\{g_k\}_{k\geq 0}$ should be selected such that $g_k\geq g^*$ and $g_k \rightarrow g^*$. One straightforward way to generate this sequence is by applying an accelerated projected gradient method to the lower-level objective $g$ separately. The loss function of the iterates generated by this algorithm can be considered as $\{g_k\}$ for the above halfspace. Note that in this case, it is known that 
\begin{equation}\label{eq:convergence_g}
   0\leq  g_k - g^* \leq \frac{2L_g\|\bx_0-\bx^*\|^2}{(k+1)^2},\quad \forall k \geq 0.
\end{equation}
Hence, the above requirements on the sequence $\{g_k\}$ are satisfied. Two remarks on the set $\cX_k$ are in order. 
First, the set $\mathcal{X}_k$ in \eqref{eq:X_k} has an explicit form, and thus computing the Euclidean projection onto $\mathcal{X}_k$ is tractable.
It also can be verified that the set $\mathcal{X}_k$ always contains the lower-level problem solution set $\mathcal{X}_g^*$. To prove this, let $\hat{\bx}^*$ be any point in $\mathcal{X}_g^*$. By using the convexity of $g$, we obtain that $ g(\by_k) + \fprod{\nabla g(\by_k),\hat{\bx}^*-\by_k} \leq g(\hat{\bx}^*)=g^* \leq g_k$. Thus $\hat{\bx}^*$ indeed satisfies both constraints in \eqref{eq:X_k} and we obtain that $\hat{\bx}^* \in \mathcal{X}_k$.  

Now that we have identified an appropriate replacement for the set $\mathcal{X}^*_g$, we can easily implement a variant of the projected accelerated gradient method for the bilevel problem using the surrogate set $\mathcal{X}_k$. We refer to our method as 
the Accelerated Gradient Method for Bilevel Optimization (AGM-BiO) and its steps are outlined in Algorithm~\ref{alg:AGM-BiO}. It is important to note that the iterates, when projected onto the set $\mathcal{X}_k$, may not belong to the set $\mathcal{X}^*_g$, as $\mathcal{X}_k$ is an approximation of the true solution set. Consequently, the iterates might be infeasible. However, the design of $\mathcal{X}_k$ allows us to control the infeasibility of the iterates, as we will demonstrate in the convergence analysis section.

\begin{remark}

The design of the halfspace as specified in \eqref{eq:X_k} should be recognized as a nuanced task. Various alternative formulations of halfspaces could fulfill the same primary conditions,  such as $\{\bz \in \mathbb{R}^d: g(\bx_k) + \langle\nabla g(\bx_k), \bz-\bx_k\rangle \leq g_k\}$ and $\{\bz \in \mathbb{R}^d: g(\bz_k) + \langle\nabla g(\bz_k), \bz-\bz_k\rangle \leq g_k\}$. However, the selection of the gradient at $\by_k$ for constructing the halfspace is not arbitrary but essential as we characterize in the convergence analysis of our method.
\end{remark}
\begin{remark}
While our paper focuses on the smooth setting, our algorithm can be extended to the composite setting with an additional assumption and some proof modifications. Due to limited space, please refer to Section~\ref{sec:extension} in the Appendix for more details.
\end{remark}

\section{Convergence Analysis}\label{sec:analysis}
In this section, we analyze the convergence rate and iteration complexity of our proposed AGM-BiO method for convex simple bilevel optimization problems. We choose the stepsize $a_{k} = (k+1)/(4L_{f})$, which is inspired from our theoretical analysis. The main theorem is as follows,
\begin{theorem}\label{thm:upper_lower}
Suppose Assumption~\ref{ass:1} holds. Let $\{\bx_k\}_{k\geq 0}$ be the sequence of iterates generated by Algorithm~\ref{alg:AGM-BiO} with stepsize $a_k = (k+1)/(4L_f)$ for
$k \geq 0$ and suppose the sequence $g_k$ used for generating the cutting plane satisfies \eqref{eq:convergence_g}. Then, for any $k\geq 0$ we have,
\begin{enumerate}[label=(\roman*)]
\vspace{-2.5mm}
    \item The function suboptimality is bounded above by 
    \begin{equation}
        f(\bx_k)-f(\bx^*) \leq \frac{4L_f \|\bx_0-\bx^*\|^2}{k(k+1)}
         \vspace{-1mm}
    \end{equation}
   
    \item The infeasibility term is bounded above by 
    \begin{equation}
        g(\bx_k)-g(\bx^*) \leq \frac{4L_g\|\bx_0-\bx^*\|^2 (\ln k+1)}{k(k+1)} +\frac{2L_gD^2}{k+1}
         \vspace{-1mm}
    \end{equation}

    \item Furthermore, if the condition $f(\bx_k) \geq f(\bx^*)$ holds, then the infeasibility term is bounded above by
    \begin{equation}
        g(\bx_k)-g(\bx^*) \leq \frac{8L_g\|\bx_0-\bx^*\|^2(\ln k+1)}{k(k+1)}.
    \end{equation}
\end{enumerate}

\vspace{-1.5mm}
\end{theorem}

Theorem~\ref{thm:upper_lower} shows the upper-level objective function gap is upper bounded by $\mathcal{O}(1/k^{2})$, which matches the convergence rate of the accelerated gradient method for single-level optimization problems. On the other hand, the suboptimality of the lower-level objective which measures infeasibility for the bilevel problem in the worst case is bounded above by $\mathcal{O}(1/k)$. In cases where the $f(\bx_k)\geq f^*$ this upper bound could be even improved to $\mathcal{O}(1/k^2)$. As a corollary of the worst-case bounds, Algorithm~\ref{alg:AGM-BiO} will return an $(\epsilon_f , \epsilon_g)$-optimal solution after at most the following number of iterations
$\mathcal{O}\left(\max \left\{\frac{1}{\sqrt{\epsilon_{f}}}, \frac{1}{\epsilon_{g}}\right\}\right)$.
We should emphasize that, under the assumptions being considered, this complexity bound represents the best-known bound among all previous works summarized in Table~\ref{sum}. It is worth noting that concurrent work by \cite{samadi2023achieving} achieves $\mathcal{O}(\max \{\frac{1}{\epsilon_{f}}, \frac{1}{\epsilon_{g}}\})$ under similar assumptions. They can only improve their complexity bound with the additional assumption that $g$ satisfies the weak sharpness condition.

\begin{remark}[The necessity of compactness of $\mathcal Z$]\label{remark:why_compact}
For the lower-level objective, we show that $A_k(g(\bx_k)-g(\bx^*)) \leq \sum_{i=0}^{k-1} a_i(g_i-g^*) + \frac{L_g}{4L_f}\sum_{i=0}^{k-1} \|\bz_{i+1}-\bz_i\|^2$ (\eqref{eq:summing_lower_level} in Section~\ref{sec:analysis}). The main challenge in obtaining an accelerated rate of $\mathcal{O}(1/k^2)$ for $g$ is controlling $\sum_{i=0}^{k-1} \|\bz_{i+1}-\bz_i\|^2$. Without a lower bound on $f$, this term cannot be bounded by the upper-level suboptimality alone. If $f(\bx_{k}) \geq f(\bx^{*})$, we can achieve the rate of $\mathcal{O}(1/k^{2})$ for $g$. Otherwise, we use the compactness of $\cZ$ to achieve $\mathcal{O}(1/k)$ for $g$. Please refer to Section~\ref{sec:analysis} for more details.
\end{remark}

\begin{remark}[Removable $\log$ terms]\label{remark:log}
The $\log$ terms in all the complexity results can be removed by choosing the auxiliary sequence $g_k = g_K$ for all $0 \leq k \leq K$, which satisfies the condition \eqref{eq:convergence_g}. This eliminates the $\log$ term in \eqref{eq:summing_lower_level} and all subsequent results. However, this choice of $\{g_{k}\}_{k \geq 0}$ requires predetermining the total number of iterations $K$.
\end{remark}

Since the algorithm's output $\hat{\bx}$ may fall outside the feasible set $\mathcal{X}_{g}^{*}$, the expression $f(\hat{\bx})-f^*$ may not necessarily be non-negative. On the other hand, under the considered assumptions, proving convergence in terms of $|f(\hat{\bx})-f^*|$ is known to be impossible due to a negative result presented by \cite{chen2023bilevel}. Specifically, for any first-order method and a given number of iterations $k$,  they demonstrated the existence of an instance of  Problem~\eqref{eq:bi-simp} where $|f(\bx_{k})-f^*| \geq 1$ for all $k\geq 0$. Thus, to provide any form of guarantee in terms of the absolute value of the suboptimality, i.e., $|f(\hat{\bx})-f^*|$, we need an additional assumption to obtain a lower bound on suboptimality and to provide a convergence bound for $|f(\hat{\bx})-f^*|$. We will address this point in the following section.

\subsection{Convergence under Hölderian Error Bound}
In this section, we introduce an additional regularity condition on $g$ to establish a lower bound for  $f(\hat{\bx})-f^*$. Formally, we assume that the lower-level objective function, $g$, satisfies the Hölderian Error Bound condition. This condition governs how the objective value $g(\bx)$ grows as the point $\bx$ moves away from the optimal solution set $\cX_{g}^*$. Intuitively, since our method's output $\hat{\bx}$ is $\epsilon_g$-optimal for the lower-level problem, it should be close to the optimal solution set $\cX_{g}^*$ when this regularity condition on $g$ is met.  Consequently, we can utilize this proximity to establish a lower bound for $f(\hat{\bx})-f^*$  by leveraging the smoothness property of  $f$. 

\begin{assumption}\label{ass:holderian}
    The function $g$ satisfies the Hölderian error bound for some $\alpha>0$ and $r \geq 1$, i.e,
    \begin{equation}\label{eq:holderian}
    \frac{\alpha}{r} \operatorname{dist}\left(\mathbf{x}, \mathcal{X}_g^*\right)^r \leq g(\mathbf{x})-g^*, \quad \forall \mathbf{x} \in \mathcal{Z},
    \end{equation}
    where $\operatorname{dist}\left(\mathbf{x}, \mathcal{X}_g^*\right) \triangleq \inf _{\mathbf{x}^{\prime} \in \mathcal{X}_g^*}\left\|\mathbf{x}-\mathbf{x}^{\prime}\right\|$.
\end{assumption}

We note that the Hölderian error bound condition in (\ref{eq:holderian}) is well-studied in the optimization literature \cite{pang1997error,bolte2017error,roulet2017sharpness} and is known to hold in general when function $g$ is analytic and the set $\cZ$ is bounded \cite{luo1994error}. There are two important special cases of the Hölderian error bound condition: 1) $g$ satisfies (\ref{eq:holderian}) with $r = 1$ known as the weak sharpness condition \cite{burke1993weak,burke2005weak}; 2) $g$ satisfies (\ref{eq:holderian}) with $r = 2$ known as the quadratic functional growth condition \cite{drusvyatskiy2018error}.

{By using the Hölderian error bound condition, \citet{jiangconditional} established a stronger relation between suboptimality and infeasibility, as shown in the following proposition.} 

\begin{proposition}[{\cite[Proposition 1]{jiangconditional}}]\label{pp:holderian}
    Assume that $f$ is convex and $g$ satisfies Assumption~\ref{ass:holderian}, and define $M=\max _{\mathbf{x} \in \mathcal{X}_g^*}\|\nabla f(\mathbf{x})\|_*$. Then  
    it holds that
    $f(\hat{\mathbf{x}})-f^* \geq -M(\frac{r (g(\hat{\mathbf{x}})-g^*)}{\alpha})^{\frac{1}{r}}$ for any $\hat{\mathbf{x}} \in \mathcal{Z}$.

\end{proposition}
Hence, under Assumption~\ref{ass:holderian}, Proposition~\ref{pp:holderian} shows that the suboptimality  $f({\hat{\bx}})-f^*$ can also be bounded from below when $\hat{\bx}$ is an approximate solution of the lower-level problem. As a result,  we can establish a convergence bound on $|f(\bx_k)-f^*|$ by combining Proposition~\ref{pp:holderian} with the upper bounds in Theorem~\ref{thm:upper_lower}. 
Moreover, it also allows us to improve the convergence rate for the lower-level problem. To prove this claim, we first introduce the following lemma which establishes an upper bound on the weighted sum of upper and lower-level objectives.
\begin{lemma}\label{lm:weighted_sum}
Suppose conditions (ii) and (iii) in  Assumption~\ref{ass:1} hold. Let $\{\bx_k\}$ be the sequence of iterates generated by Algorithm~\ref{alg:AGM-BiO} with stepsize
    $a_k = \gamma(k+1)/(4L_f)$, where $0<\gamma\leq1$. Moreover, suppose the sequence $g_k$ used for generating the cutting plane satisfies satisfies \eqref{eq:convergence_g}. Then, for any $\lambda \geq \frac{L_g}{(2/\gamma-1)L_f}$ and $k\geq 0$ we have 
    \begin{equation}\label{eq:weighted_sum}
    \begin{aligned}
       \lambda(f(\bx_{k})-f(\bx^*)) + g(\bx_{k})-g(\bx^*) 
       \leq \frac{4L_g\|\bx_0-\bx^*\|^2(\ln k+1)}{k(k+1)}   + \frac{4\lambda L_{f}\|\bx_{0}-\bx^*\|^2}{\gamma k(k+1)}
    \end{aligned}
    \end{equation}
\end{lemma}

This result characterizes and upper bound of $\mathcal{O}(1/k^2)$ on the expression $\lambda(f(\bx_{k})-f(\bx^*)) + g(\bx_{k})-g(\bx^*)$. That said, the first term in this expression, a.k.a., $\lambda(f(\bx_{k})-f(\bx^*))$ may not be non-negative for a bilevel problem as discussed earlier. Hence, we cannot simply eliminate $\lambda(f(\bx_{k})-f(\bx^*))$ to show an upper bound of $\mathcal{O}(1/k^2)$ on infeasibility, a.k.a., $g(\bx_{k})-g(\bx^*)$.
Instead, we leverage the Hölderian error bound on $g$ and apply Proposition~\ref{pp:holderian} to the first term. As a result, we can eliminate the dependence on $f$ in \eqref{eq:weighted_sum}. In this case, we can establish an upper bound on infeasibility.

\begin{theorem}\label{thm:upper_lower_hd}
Suppose conditions (ii) and (iii) in Assumption~\ref{ass:1} hold and the lower-level function $g$ satisfies the Hölderian error bound with $r\!> \!1$. Let $\{\bx_k\}$ be the iterates generated by Algorithm~\ref{alg:AGM-BiO} with stepsize $a_k = \gamma\frac{k+1}{4L_f}$, where $\gamma = 1/(\frac{2L_g}{L_f}T^{\frac{2r-2}{2r-1}}+2)$ and $T$ is the total number of iterations that we run the algorithm. Moreover, suppose the sequence $g_k$ used for generating the cutting plane satisfies \eqref{eq:convergence_g}. If we define the constants $C_f \triangleq 8 L_f \|\bx_0-\bx^*\|^2$, $C_g \triangleq 12 L_g \|\bx_0-\bx^*\|^2$ and  $C \triangleq M(\frac{r}{\alpha})^{\frac{1}{r}}$, where $M \triangleq \max _{\mathbf{x} \in \mathcal{X}_g^*}\|\nabla f(\mathbf{x})\|$, $\alpha$ and $r$ are the parameters in Assumption~\ref{ass:holderian}, then the following results hold:
    \begin{enumerate}[label=(\roman*)]
    \vspace{-3mm}
        \item The function suboptimality is bounded above by \begin{equation}\label{eq:with_holderian_r>1}
            f(\bx_T)-f(\bx^*) \leq \frac{C_g (\ln T + 1)}{T^{2r/(2r-1)}} + \frac{C_f}{T^{2}}
        \end{equation}

            \vspace{-3mm}

        \item The function suboptimality is bounded below by 
        \begin{align}
            f(\bx_T) - f(\bx^*) 
            \geq\!- C\max\!\left\{\!\frac{(2C_{\!g} (\ln T\! +\! 1))^{\frac{1}{r}}}{T^{\frac{2}{2r-1}}}\! +\! \frac{(2C_{\!f})^{\frac{1}{r}}}{T^{\frac{2}{r}}} ,\frac{(2C)^{\frac{1}{r-1}}}{T^{\frac{2}{2r-1}}}\!\right\}   \nonumber
        \end{align}

            \vspace{-3mm}

        \item The infeasibility term is bounded above by 
        \begin{equation}
        \begin{aligned}
            g(\bx_T)-g(\bx^*) 
            \leq \max \left\{\frac{2C_g (\ln T + 1)}{T^{2r/(2r-1)}} + \frac{2C_f}{T^{2}},\frac{(2C)^{\frac{r}{r-1}}}{T^{\frac{2r}{2r-1}}}\right\}
        \end{aligned}
        \end{equation}
    \end{enumerate}
\end{theorem}
Before unfolding this result, we would like to highlight that unlike the result in Theorem~\ref{thm:upper_lower}, the above bounds in Theorem~\ref{thm:upper_lower_hd} do not require the feasible set to be compact. Since $r>1$, the first result shows $f(\bx_T)-f(\bx^*)$  has an upper bound of $\mathcal{O}((\frac{1}{T})^{\frac{2r}{2r-1}})$
and the second result guarantees a lower bound of $-\mathcal{O}((\frac{1}{T})^{\frac{2}{2r-1}})$. These two bounds together lead to an upper bound of $\mathcal{O}((\frac{1}{T})^{\frac{2}{2r-1}})$ for the absolute error $|f(\bx_T)-f(\bx^*)|$. Moreover, the third result implies that the lower-level problem suboptimality which measures infeasibility is bounded above by $\mathcal{O}((\frac{1}{T})^{\frac{2r}{2r-1}})$.  


The previous result presented in Theorem~\ref{thm:upper_lower_hd} is applicable when $r>1$. However, for the case that $1$st-order Hölderian error bound condition on $g$ holds (i.e., weak sharpness condition), we require a distinct analysis and a different choice of $\gamma$ to achieve the tightest bounds. In the subsequent theorem, we present our findings for this specific scenario.

\begin{theorem}\label{pp:upper_lower_ws}
Suppose conditions (ii) and (iii) in Assumption~\ref{ass:1} are met and that the lower-level objective function $g$ satisfies the Hölderian error bound with $r = 1$. Let $\{\bx_k\}$ be the sequence of iterates generated by Algorithm~\ref{alg:AGM-BiO} with stepsize $a_k = \gamma\frac{k+1}{4L_f}$, where  $0<\gamma\leq \min\{\frac{2\alpha L_f}{2ML_g + \alpha L_f},1\}$. Moreover, suppose the sequence $g_k$ used for generating the cutting plane satisfies \eqref{eq:convergence_g}, and recall $M \triangleq \max _{\mathbf{x} \in \mathcal{X}_g^*}\|\nabla f(\mathbf{x})\|$ and $\alpha$ in Assumption~\ref{ass:holderian}. If we define the constants $C_f 
 \triangleq 4L_f\|\bx_{0}-\bx^*\|^2$ and  $C_g \triangleq 8L_g\|\bx_0-\bx^*\|^2$, then for any $k\geq 0$:
\vspace{-3mm}
 \begin{enumerate}[label=(\roman*)]
 \item The function suboptimality is bounded above by
 \begin{equation}
      f(\bx_{k})-f(\bx^*) \leq \frac{C_f}{\gamma k(k+1)}.
 \end{equation}

 \vspace{-3mm}
 \item The function suboptimality is bounded below by
  \begin{equation}
    f(\bx_k) - f(\bx^*)\geq -\frac{C_gM(\ln k+1)}{\alpha k(k+1)}  -\frac{C_f}{\gamma k(k+1)}.
    \end{equation}

    \vspace{-3mm}
 \item The infeasibility term is bounded above by
   \begin{equation}
        g(\bx_{k})-g(\bx^*) \leq \frac{C_g(\ln k+1)}{k(k+1)}   + \frac{\alpha C_f}{\gamma Mk(k+1)}.
    \end{equation}
 \end{enumerate}
\end{theorem}

Theorem~\ref{pp:upper_lower_ws} shows that under the Hölderian error bound with $r = 1$, also known as weak sharpness condition, the absolute value of the function suboptimality $|f(\bx_k) - f(\bx^*)|$  approaches zero at a rate of $\mathcal{O}(1/k^2)$ -- ignoring the log term. The lower-level error $g(\bx_{k})-g(\bx^*) $, capturing the infeasibility of the iterates, also approaches zero at a rate of $\mathcal{O}(1/k^2)$. 
As a corollary, Algorithm~\ref{alg:AGM-BiO} returns an $(\epsilon_f , \epsilon_g)$-absolute optimal solution after $
    \mathcal{O}(\max \{\frac{1}{\sqrt{\epsilon_{f}}}, \frac{1}{\sqrt{\epsilon_{g}}}\})$ iterations.
This iteration complexity matches the result in a concurrent work \cite{samadi2023achieving} under similar assumptions.

\section{Numerical Experiments}\label{sec:experiment}

In this section, we evaluate our AGM-BiO method on two different bilevel problems using real and synthetic datasets. We compare its runtime and iteration count with other methods,  including a-IRG \cite{kaushik2021method}, CG-BiO \cite{jiangconditional}, Bi-SG \cite{merchav2023convex}, SEA \cite{shen2023online}, R-APM \cite{samadi2023achieving}, PB-APG \cite{chen2024penalty}, and Bisec-BiO \cite{wang2024near}. 


\begin{figure*}[t!]
  \centering
  \subfloat[Infeas. vs. time]
{\includegraphics[width=0.25\textwidth]{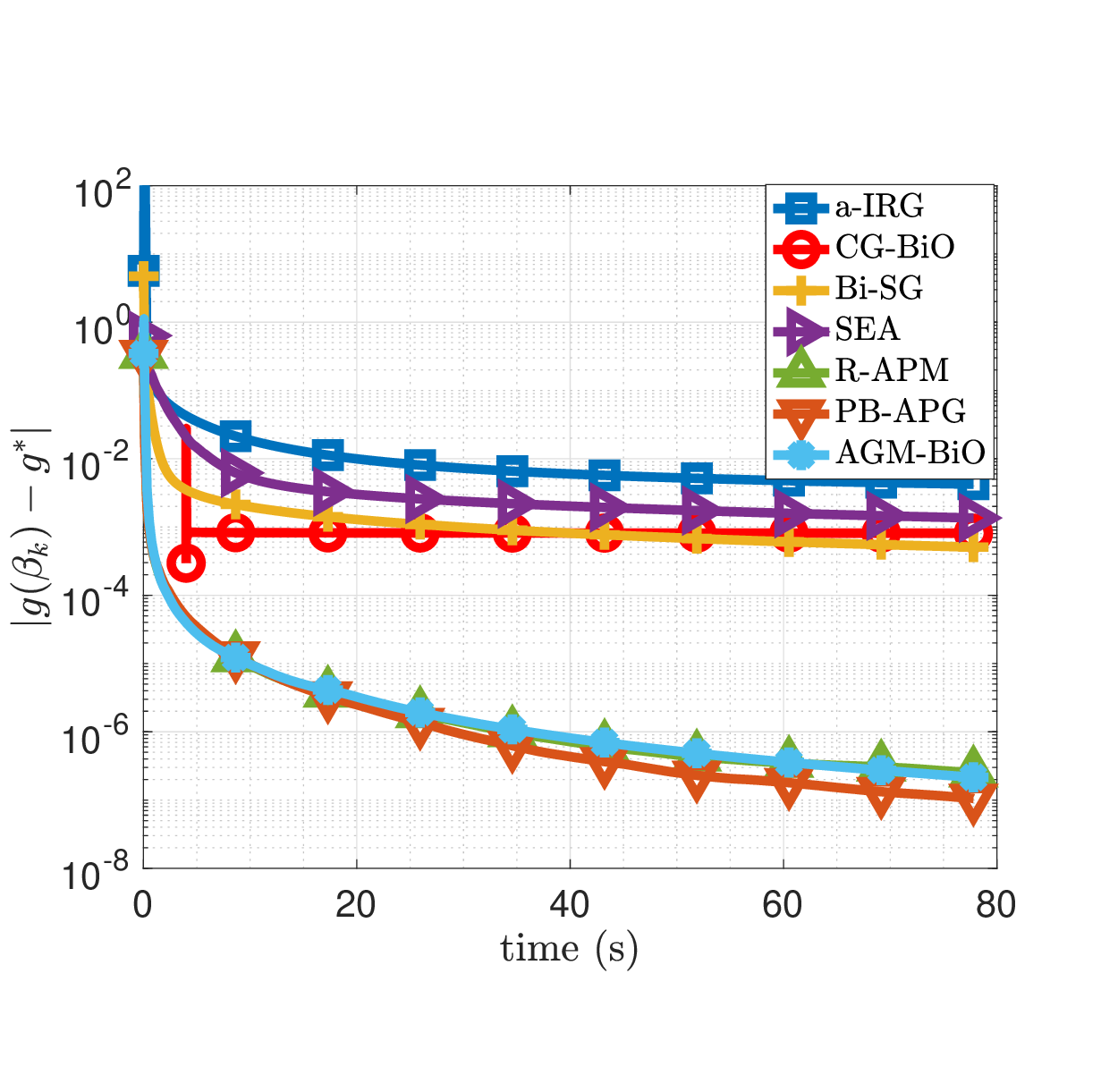}}
  \hspace{-2mm}
  \subfloat[Subopt. vs. time]
  {\includegraphics[width=0.25\textwidth]{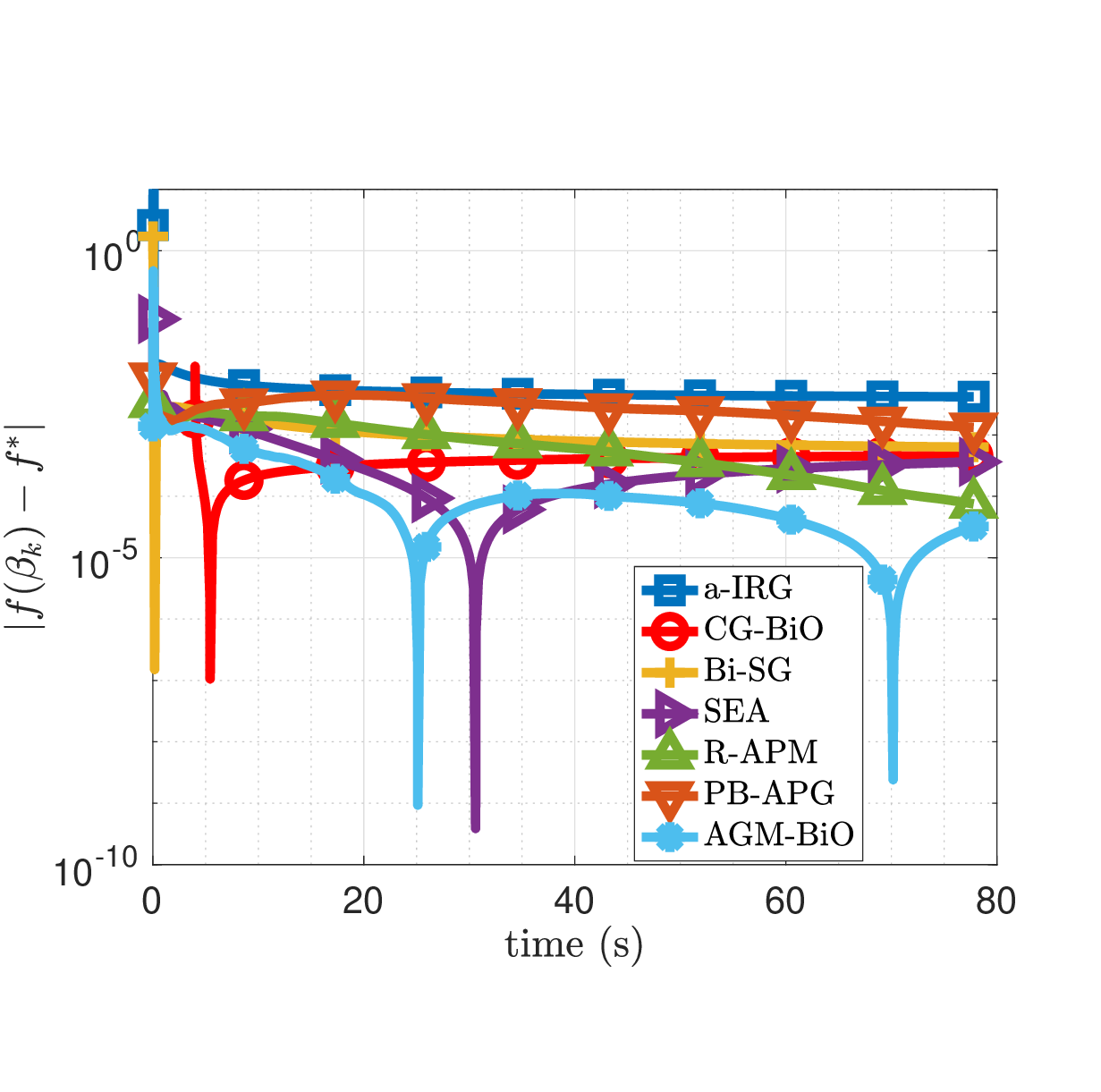}}
  \hspace{-2mm}
  \subfloat[Infeas. vs. iterations]
{\includegraphics[width=0.25\textwidth]{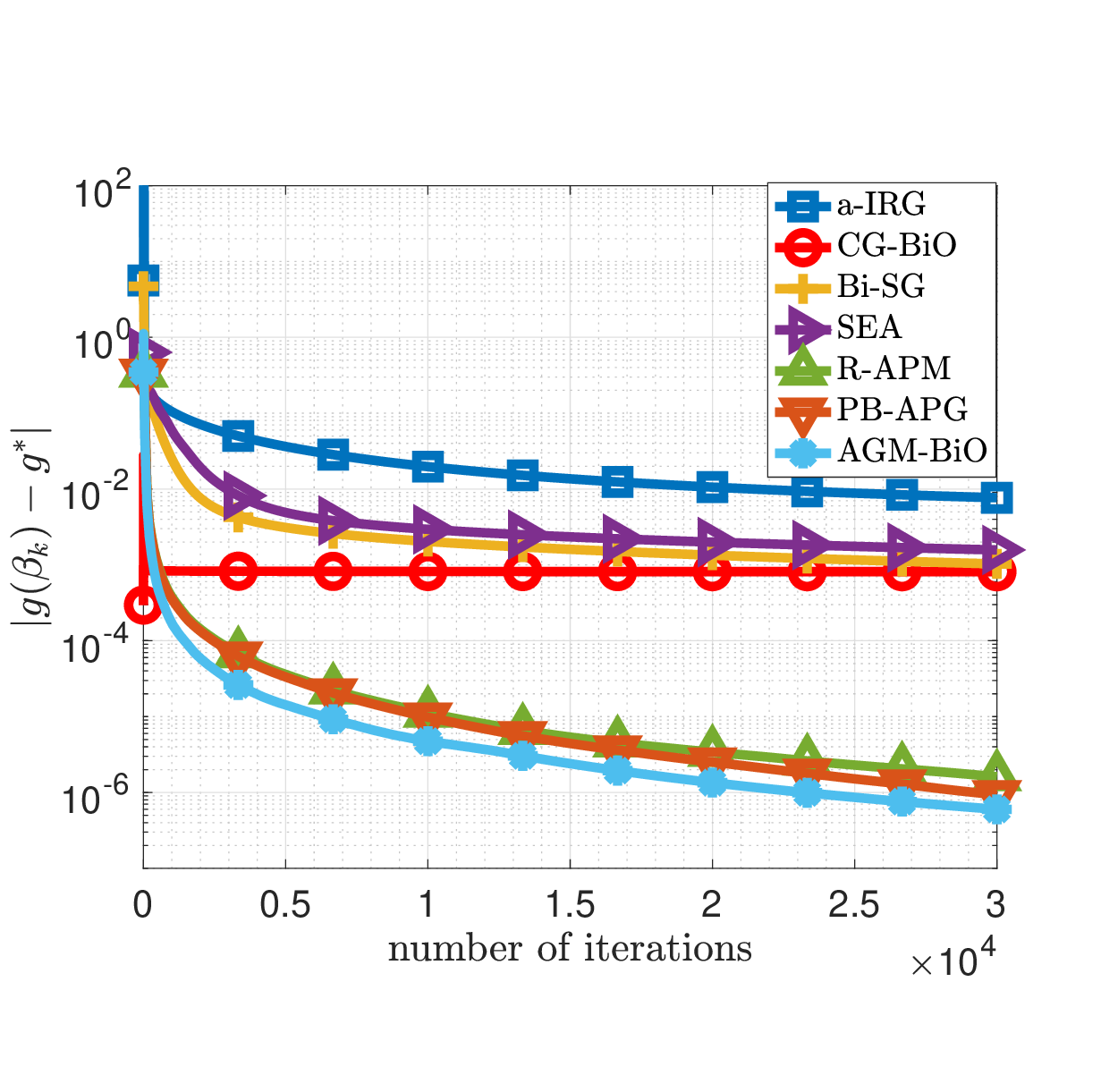}}
  \hspace{-2mm}
  \subfloat[Subopt. vs. iterations]
  {\includegraphics[width=0.25\textwidth]{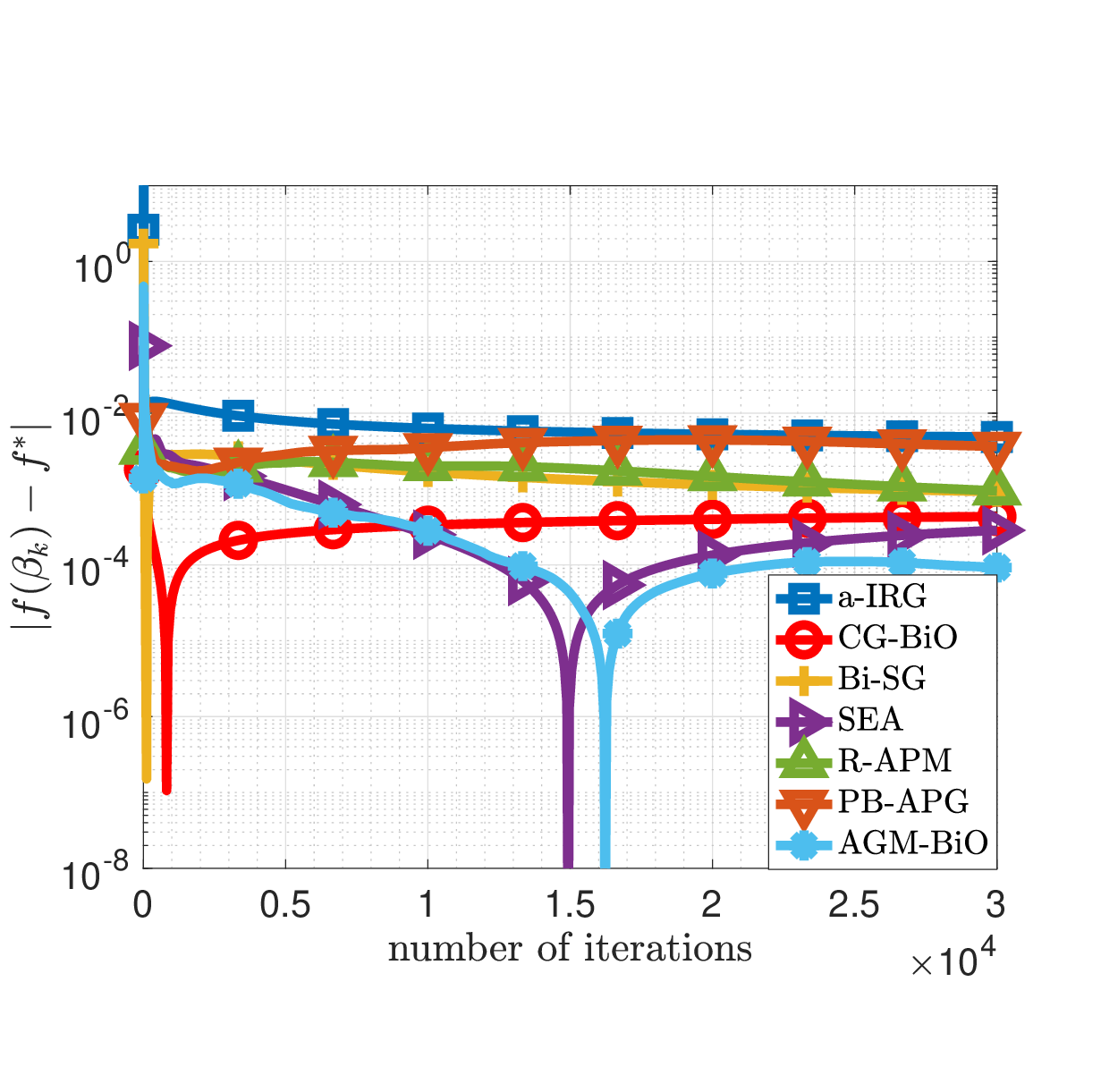}}
  \vspace{0mm}
  \caption{Comparison of a-IRG, CG-BiO, Bi-SG, SEA, R-APM, PB-APG, and AGM-BiO for solving the over-parameterized regression problem.}
  \label{fig:regression}
\end{figure*}

\noindent\textbf{Over-parameterized regression.}
We examine problem \eqref{eq:bi-simp} where the lower-level problem corresponds to training loss, and the upper-level pertains to validation loss. The objective is to minimize the validation loss by selecting an optimal training loss solution. This method is also referred to as lexicographic optimization \cite{gong2021bi}.
 A common example of that is the constrained regression problem, where we aim to find an optimal parameter vector $\boldsymbol{\beta} \in \mathbb{R}^d$ for the validation loss that minimizes the loss $\ell_{\operatorname{tr}}(\boldsymbol{\beta})$ over the training dataset $\mathcal{D}_{\mathrm{tr}}$. To represent some prior knowledge, we constrain $\boldsymbol{\beta}$ to be in some subset $\mathcal{Z} \subseteq \mathbb{R}^d$, e.g., $\mathcal{Z}=\left\{\boldsymbol{\beta} \mid \beta_{1}\leq\dots \leq \beta_{d}\right\}$ in isotonic regression and $\mathcal{Z}=\left\{\boldsymbol{\beta} \mid\|\boldsymbol{\beta}\|_{p} \leq \lambda\right\}$ in $L_{p}$ constrained regression. Without explicit regularization, an over-parameterized regression over the training dataset has multiple global minima, but not all these optimal regression coefficients perform equally on validation or testing datasets. Thus, the upper-level objective serves as a secondary criterion to ensure a smaller error on the validation dataset $\mathcal{D}_{\mathrm{val}}$. The problem can be cast as 
 \begin{equation*}
\min_{\boldsymbol{\beta} \in \mathbb{R}^d} f(\boldsymbol{\beta}) \triangleq \ell_{\mathrm{val}}(\boldsymbol{\beta}) \quad
\textrm{ s.t. }  \quad \boldsymbol{\beta} \in \underset{\mathbf{z} \in \mathcal{Z}}{\operatorname{argmin}}~g(\mathbf{z}) \triangleq \ell_{\mathrm{tr}}(\mathbf{z}) 
\end{equation*}
In this case, both upper-level and lower-level objectives are convex and smooth if the loss $\ell$ is smooth and convex. Since projections onto the sublevel set of $f$ are difficult to compute, Bisec-BiO is excluded from this experiment.

We apply the Wikipedia Math Essential dataset \cite{rozemberczki2021pytorch} which is composed of a data matrix $\mathbf{A} \in \mathbb{R}^{n \times d}$ with $n=1068$ samples and $d=730$ features and an output vector $\bb \in \mathbb{R}^n$.  We use  $75\%$ of the dataset as the training set $\left(\mathbf{A}_{\mathrm{tr}}, \mathbf{b}_{\mathrm{tr}}\right)$ and $25\%$ as the validation set $\left(\mathbf{A}_{\text {val}}
, \mathbf{b}_{\text {val}}\right)$. 
For both upper- and lower-level loss functions, we use the least squared loss. Then the lower-level objective is  $g(\boldsymbol{\beta}) = \frac{1}{2}\|\mathbf{A}_{tr}\boldsymbol{\beta} - \bb_{tr}\|_{2}^{2}$, the upper-level objective is $f(\boldsymbol{\beta}) = \frac{1}{2}\|\mathbf{A}_{val}\boldsymbol{\beta} - \bb_{val}\|_{2}^{2}$, and the constraint set is chosen as the unit $L_{2}$-ball $\mathcal{Z}=\left\{\boldsymbol{\beta} \mid\|\boldsymbol{\beta}\|_{2} \leq 1 \right\}$. Note that this regression problem is over-parameterized since the number of features $d$ is larger than the number of data points in both the training set and validation set. 


In Figures 1(a) and 1(c), we observe that the three accelerated gradient-based methods (R-APM, PB-APG, and AGM-BiO) converge faster in reducing infeasibility, both in terms of runtime and number of iterations. In terms of absolute suboptimality, shown in Figures 1(b) and 1(d), AGM-BiO achieves the smallest absolute suboptimality gap among all algorithms. Unlike the infeasibility plots, R-APM and PB-APG underperform compared to AGM-BiO. Note that the lower-level objective in this problem does not satisfy the weak sharpness condition, so the regularization parameter $\eta$ in R-APM is set as $1/(K+1)$. Consequently, the suboptimality for R-APM converges slower than AGM-BiO, as suggested by the theoretical results in Table~\ref{sum}.


\begin{figure*}
  \centering
    \vspace{-5mm}
  \subfloat[Infeas. ($n=3$)]
{\includegraphics[width=0.25\textwidth]{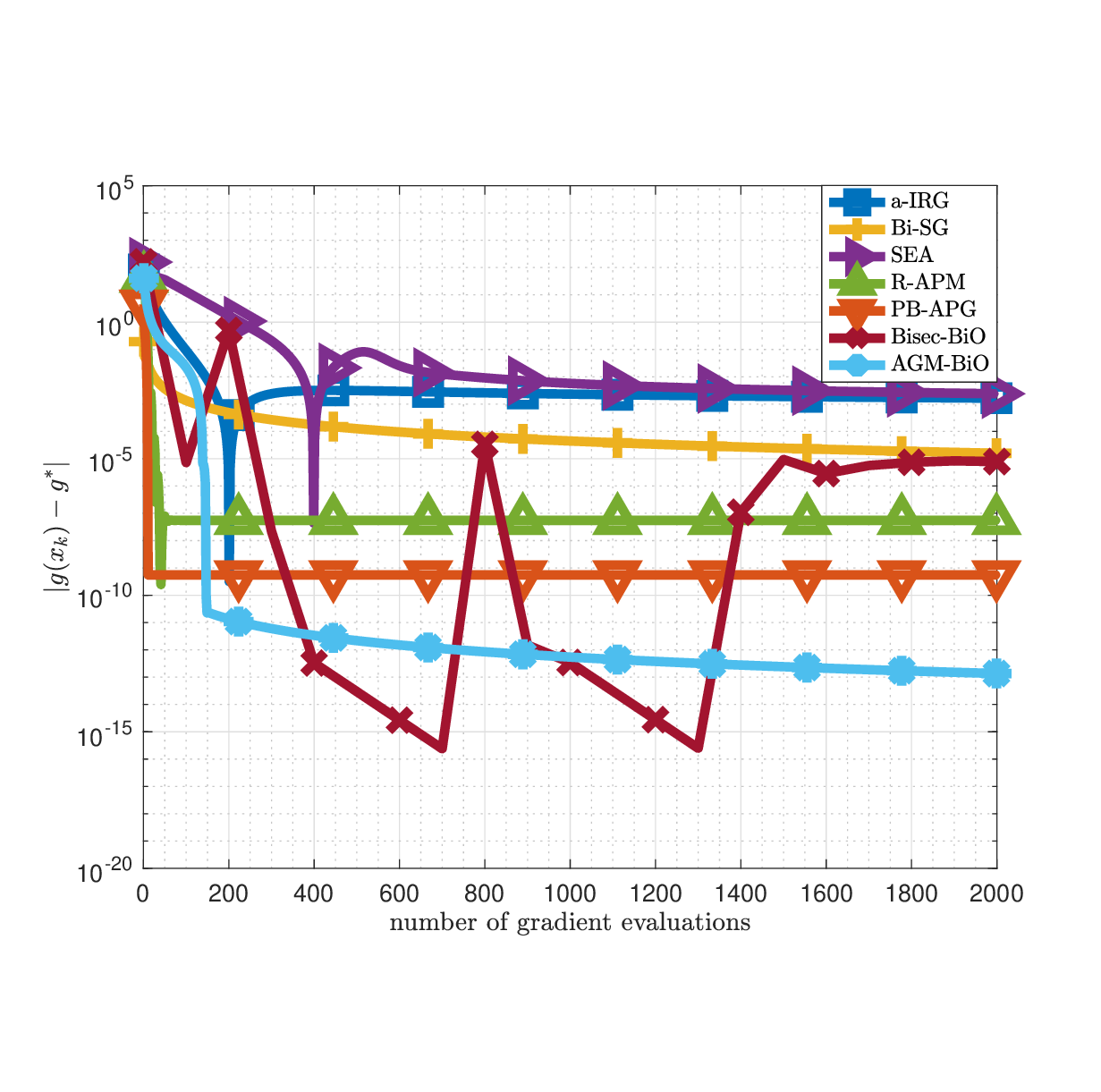}}
\hspace{-2mm}
  \subfloat[Subopt. ($n=3$)]
  {\includegraphics[width=0.25\textwidth]{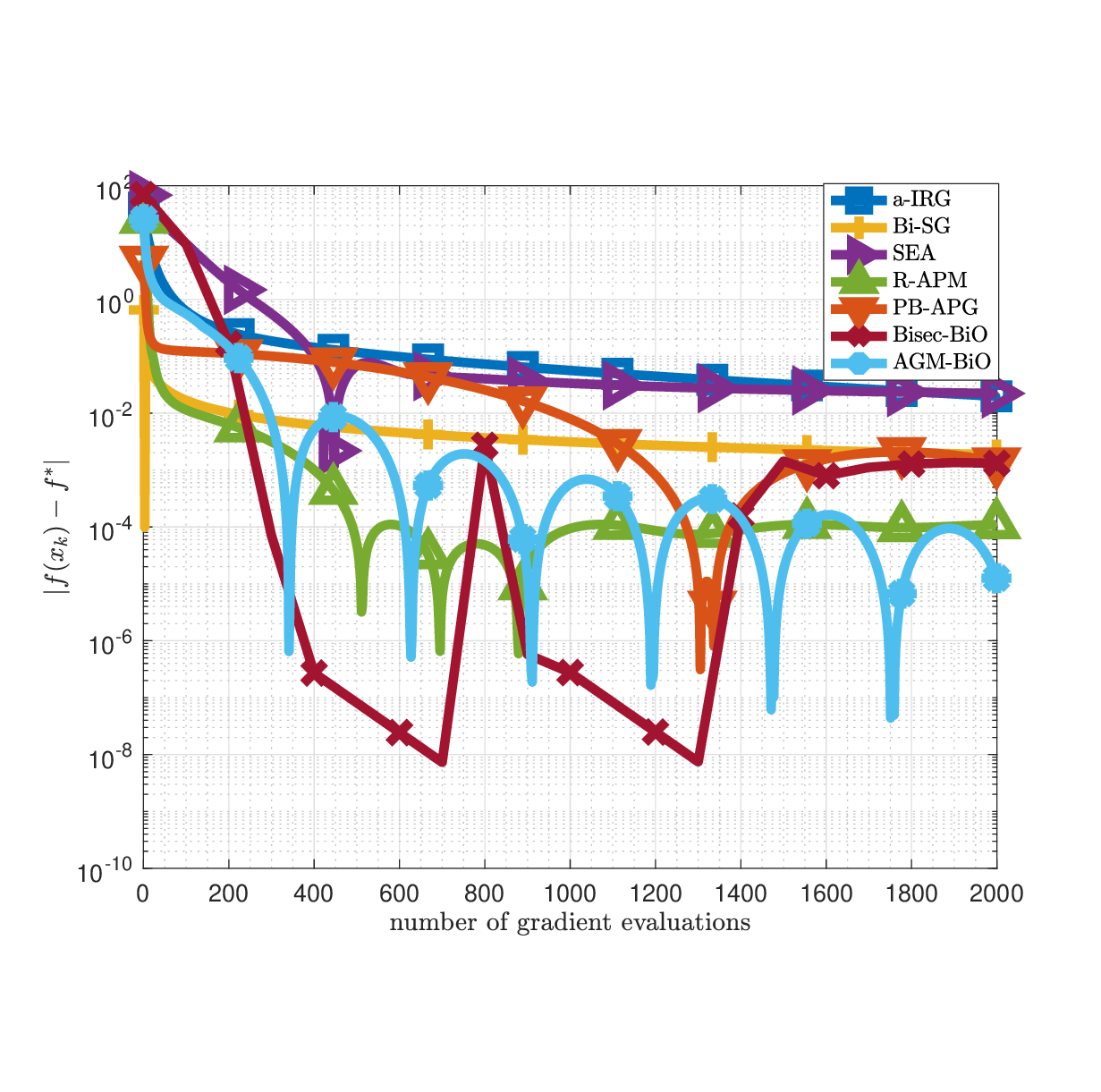}}
\hspace{-2mm}
\subfloat[Infeas. ($n=100$)]
{\includegraphics[width=0.25\textwidth]{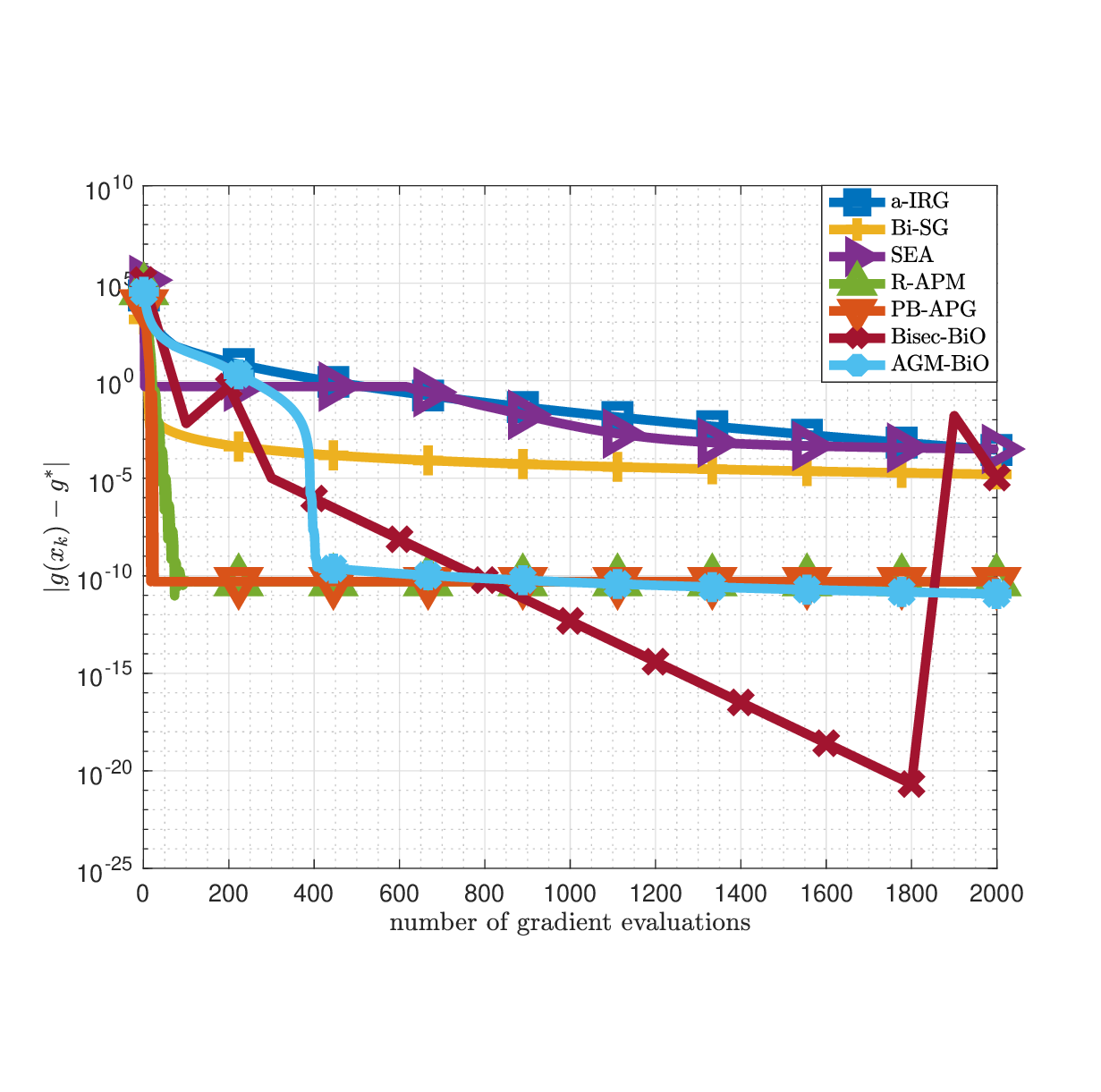}}
\hspace{-2mm}
\subfloat[Subopt. ($n=100$)]
  {\includegraphics[width=0.25\textwidth]{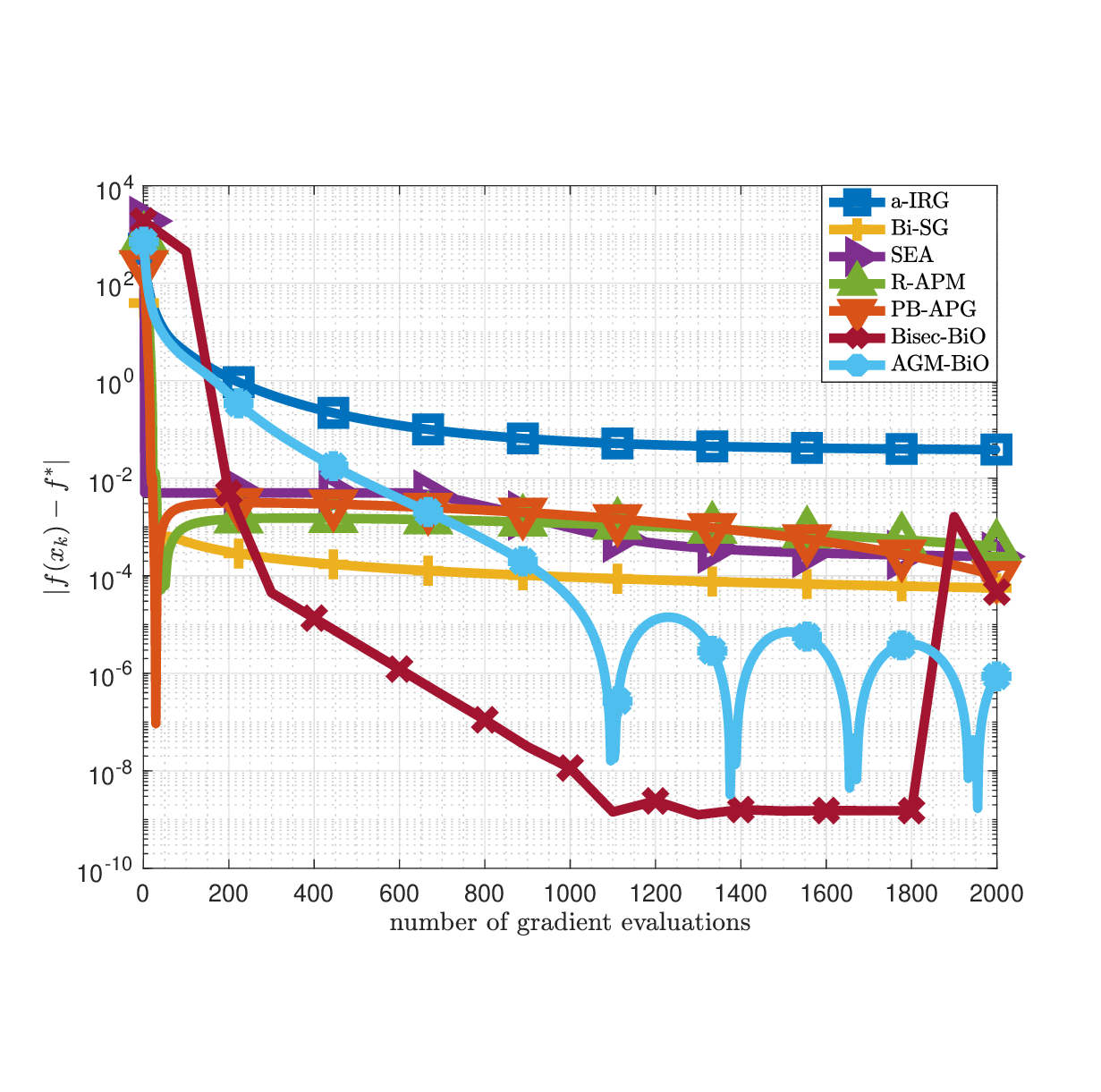}}
  \vspace{0mm}
  \caption{Comparison of a-IRG, Bi-SG, SEA, R-APM, PB-APG, Bisec-BiO, and AGM-BiO for solving the linear inverse problem.}
  \label{fig:inverse_problem}
\end{figure*}

\noindent\textbf{Linear inverse problems.} 
In the next experiment, we concentrate on a problem that fulfills the Hölderian Error Bound condition for some $r > 1$. We aim to evaluate the performance of our method in this specific context and verify the validity of our theoretical results for this scenario. Specifically, we focus on the so-called linear inverse problems, commonly used to evaluate convex bilevel optimization algorithms, which originate from \cite{sabach2017first}. 
The goal of linear inverse problems is to obtain a solution $\bx \in \mathbb{R}^{n}$ to the system of linear equation $\bA\bx = \bb$.  Note that if $\bA$ is rank-deficient, there can be multiple solutions, or there might be no exact solution due to noise. To address this issue, we chase a solution that has the smallest weighted norm with respect to some positive definite matrix $\bQ$, i.e., $\|\bx\|_\bQ:= \sqrt{\bx^\top \bQ\bx} $. This problem can be also cast as the following simple bilevel problem: 
\begin{equation*}
\min_{\boldsymbol{\bx} \in \mathbb{R}^n} f(\boldsymbol{\bx}) \triangleq \frac{1}{2}\|\bx\|_{\bQ}^{2} \quad
\textrm{ s.t. }  \quad  \boldsymbol{\bx} \in \underset{\mathbf{z} \in \mathcal{Z}}{\operatorname{argmin}} g(\bz) \triangleq  \frac{1}{2}\|\bA\bz - \bb\|_{2}^{2}
\end{equation*}
For this class of problem, if $\bQ$, $\bA$, and $\bb$ are generated randomly or by the “regularization tools” like \cite{sabach2017first,shen2023online}, we are not able to obtain the exact optimal value $f^*$. To the best of our knowledge, no existing solver could obtain the exact optimal value $f^*$ for this bilevel problem. Specifically, the existing solvers either fail to solve this bilevel problem or return an inaccurate solution by solving a relaxed version of the problem. Hence, in~\cite{sabach2017first,shen2023online} they only reported the upper-level function value. 
However, in this paper, we intend to obtain the complexity bounds for finding $(\epsilon_{f},\epsilon_{g})$-optimal and  $(\epsilon_{f},\epsilon_{g})$-absolute optimal solutions. Without 
knowing $f^*$, we can not characterize the behavior of $|f(\bx_{k}) - f^*|$. Therefore, we choose an example where we can obtain the exact solution. Specifically, we set $\bQ = \bI_{n}$, $\bA = \mathbf{1}_{n}^{\top}$, $\bb=1$, and the constraint set $\mathcal{Z} = \mathbb{R}_{+}^{n}$. In this case, the optimal solution $\bx^* = \frac{1}{n}\mathbf{1}_{n}$ and optimal value $f^* = \frac{1}{2n}$. This specific example essentially involves seeking the minimum norm for an under-determined system. Note that the lower-level objective in this problem satisfies the Hölderian Error Bound condition with order $r = 2$ \cite{necoara2019linear}. Hence, we do not need the constraint set $\mathcal{Z}$ to be compact as shown in Theorem~\ref{thm:upper_lower_hd}. Due to the unbounded nature of the constraint set, Frank-Wolfe-type methods are not viable options. Consequently, we have opted not to incorporate CG-BiO in this experiment.

We explored examples with two distinct dimensions: $n=3$ and $n=100$, evaluating a total of $2000$ gradients. In Figures 2(a) and 2(c), AGM-BiO shows superior performance in terms of infeasibility. In Figures 2(b) and 2(d), we compare methods in terms of absolute error of suboptimality. The gap between R-APM and AGM-BiO is smaller for $n=3$, but for $n=100$, AGM-BiO significantly outperforms all other methods, including R-APM. 
Since the regularization and penalty parameters in R-APM and PB-APG are fixed, they might get stuck at a certain accuracy level, as seen in Figures 2(a) and 2(c). In contrast, AGM-BiO uses a dynamic framework for minimizing the upper and lower-level functions, consistently reducing both suboptimality and infeasibility. 
Although Bisec-BiO theoretically has the best complexity results due to the ease of projecting onto the sublevel set of $f$, its performance in the last iteration is inconsistent, as shown in Figure 2.

\section{Conclusion}

 In this paper, we introduced an accelerated gradient-based algorithm for solving a specific class of bilevel optimization problems with convex objective functions in both the upper and lower levels. Our proposed algorithm achieves a computational complexity of $\mathcal{O}(\max\{\epsilon_{f}^{-0.5},\epsilon_{g}^{-1}\})$. When an additional weak sharpness condition is applied to the lower-level function $g$, the iteration complexity improves to $\tilde{\mathcal{O}}(\max\{\epsilon_{f}^{-0.5},\epsilon_{g}^{-0.5}\})$, matching the well-known fastest convergence rate for single-level convex optimization problems. We further extended this result to an iteration complexity of $\tilde{\mathcal{O}}(\max\{\epsilon_{f}^{-\frac{2r-1}{2r}},\epsilon_{g}^{-\frac{2r-1}{2r}}\})$ when the lower-level loss satisfies the Hölderian error bound assumption.

\section*{Acknowledgements}
The research of J. Cao, R. Jiang and A. Mokhtari is supported in part by NSF Grants 2127697, 2019844, and  2112471,  ARO  Grant  W911NF2110226,  the  Machine  Learning  Lab  (MLL)  at  UT  Austin, and the Wireless Networking and Communications Group (WNCG) Industrial Affiliates Program. The research of E. Yazdandoost Hamedani is supported by NSF Grant 2127696.

\printbibliography

\newpage
\appendix
\section*{Appendix}

\section{Proof of the Main Results}
\subsection{Proof of Theorem~\ref{thm:upper_lower}}
To prove Theorem~\ref{thm:upper_lower}, we start with the following general lemma that holds for any choice of the step sizes $\{a_k\}$.
\begin{lemma}\label{lem:potential}
Let $\{\bx_k\}$ be the sequence of iterates generated by Algorithm~\ref{alg:AGM-BiO} with stepsize $a_k>0$ for
$k \geq 0$. Then we have 

\begin{equation}
\begin{aligned}
    A_{k+1}(f(\bx_{k+1})-f(\bx^*)) + \frac{1}{2}\|\bz_{k+1}-\bx^*\|^2 - \Bigl( A_{k}(f(\bx_{k})-f(\bx^*)) + \frac{1}{2}\|\bz_{k}-\bx^*\|^2\Bigr) \\
    \leq \left(\frac{L_f a_k^2}{2A_{k+1}}-\frac{1}{2}\right)\|\bz_{k+1}-\bz_k\|^2 \label{eq:potential_f}
\end{aligned}
\end{equation}
\begin{gather}  
   A_{k+1}(g(\bx_{k+1})-g(\bx^*)) - A_{k}(g(\bx_{k})-g(\bx^*)) \leq a_k (g_k-g(\bx^*)) + \frac{L_g a^2_{k}}{2A_{k+1}}\|\bz_{k+1}-\bz_k\|^2. \label{eq:potential_g}
\end{gather}
\end{lemma}
\begin{proof}[Proof of Lemma~\ref{lem:potential}]
Let $\bx^*$ be any optimal solution of \eqref{eq:bi-simp}. 
We first consider the upper-level objective $f$. Since $f$ is convex, we have 
\begin{equation}\label{eq:convex_ineq}
    f(\by_k) - f(\bx^*) \leq \fprod{\grad f(\by_k),\by_k-\bx^*},\quad 
    f(\by_k) - f(\bx_k) \leq \fprod{\grad f(\by_k),\by_k-\bx_k}.
\end{equation}
Now given the update rule $A_{k+1}=A_k+a_k$, we can write
\begin{equation}\label{eq:weighted_inequality}
    A_{k+1}(f(\by_k)-f(\bx^*)) - A_k(f(\bx_k)-f(\bx^*)) = a_k(f(\by_k) - f(\bx^*))+A_k(f(\by_k) - f(\bx_k))
\end{equation}
Combining \eqref{eq:convex_ineq} and \eqref{eq:weighted_inequality}, we have 
\begin{equation}\label{eq:anytime_online_batch}
    \begin{aligned}
    A_{k+1}(f(\by_k)-f(\bx^*)) - A_k(f(\bx_k)-f(\bx^*)) &\leq  a_k(\fprod{\grad f(\by_k),\by_k-\bx^*})+A_k(\fprod{\grad f(\by_k),\by_k-\bx_k}) \\
    &= \fprod{\grad f(\by_k), a_k \by_k +A_k(\by_k-\bx_k)-a_k\bx^*} \\
    &= a_k \fprod{\grad f(\by_k),\bz_k-\bx^*},
\end{aligned}
\end{equation}
where the last equality follows from the definition of $\by_k$. Furthermore, since $f$ is $L_f$-smooth, we have 
\begin{equation}\label{eq:f_smooth}
    f(\bx_{k+1}) \leq f(\by_k)+\fprod{\grad f(\by_k),\bx_{k+1}-\by_k}+\frac{L_f}{2}\|\bx_{k+1}-\by_k\|^2.
\end{equation}
If we multiply both sides of \eqref{eq:f_smooth} by $A_{k+1}$ and combine the resulting inequality with \eqref{eq:anytime_online_batch}, we obtain
\begin{equation}\label{eq:potential}
    \begin{aligned}
    &\phantom{{}\leq{}} A_{k+1}(f(\bx_{k+1})-f(\bx^*)) - A_k(f(\bx_k)-f(\bx^*))\\
    & \leq a_k \fprod{\grad f(\by_k),\bz_k-\bx^*}+ A_{k+1} \fprod{\grad f(\by_k),\bx_{k+1}-\by_k}+\frac{L_f A_{k+1}}{2}\|\bx_{k+1}-\by_k\|^2 \\
    & = a_k \fprod{\grad f(\by_k),\bz_k-\bx^*}+ a_{k} \fprod{\grad f(\by_k),\bz_{k+1}-\bz_k}+\frac{L_fa_k^2}{2A_{k+1}}\|\bz_{k+1}-\bz_k\|^2 \\
    & = a_k \fprod{\grad f(\by_k),\bz_{k+1}-\bx^*}+\frac{L_fa_k^2}{2A_{k+1}}\|\bz_{k+1}-\bz_k\|^2,
\end{aligned}
\end{equation}
where we used the fact that $a_k(\bz_{k+1}-\bz_k)=A_{k+1}(\bx_{k+1}-\by_k)$ in the first equality. Moreover, since $\bx^*\in \mathcal{Z}_k$, we obtain from the update rule in \eqref{eq:update_of_z} that
\begin{equation}
   \begin{aligned}\label{eq:projection}
    &\fprod{\bz_{k+1}-\bz_k+a_k\nabla f(\by_k),\bx^*-\bz_{k+1}} \geq 0\\
    \Leftrightarrow \quad & a_k\fprod{\nabla f(\by_k),\bz_{k+1}-\bx^*} \leq \fprod{\bz_{k+1}-\bz_k,\bx^*-\bz_{k+1}} \\
    \Leftrightarrow \quad & a_k\fprod{\nabla f(\by_k),\bz_{k+1}-\bx^*} \leq \frac{1}{2}\|\bz_k-\bx^*\|^2-\frac{1}{2}\|\bz_{k+1}-\bx^*\|^2-\frac{1}{2}\|\bz_{k+1}-\bz_k\|^2.
\end{aligned} 
\end{equation}
Combining \eqref{eq:potential} and \eqref{eq:projection} leads to 
\begin{equation*}
\begin{aligned}
    A_{k+1}(f(\bx_{k+1})-f(\bx^*)) + \frac{1}{2}\|\bz_{k+1}-\bx^*\|^2 - \left( A_{k}(f(\bx_{k})-f(\bx^*)) + \frac{1}{2}\|\bz_{k}-\bx^*\|^2\right) \\
    \leq \frac{1}{2}\left(\frac{L_f a_k^2}{A_{k+1}}-1\right)\|\bz_{k+1}-\bz_k\|^2,
\end{aligned}
\end{equation*}
which proves the claim in \eqref{eq:potential_f}. 

Next, we proceed to prove the claim in \eqref{eq:potential_g}. To do so, we first leverage the convexity of $g$ which leads to
\begin{equation}\label{eq:convex_ineq_g}
    g(\by_k) - g(\bx_k) \leq \fprod{\grad g(\by_k),\by_k-\bx_k}.
\end{equation}
Also, since $g$ is $L_g$-smooth, we have 
\begin{equation}\label{eq:g_smooth}
    g(\bx_{k+1}) \leq g(\by_k)+\fprod{\grad g(\by_k),\bx_{k+1}-\by_k}+\frac{L_g}{2}\|\bx_{k+1}-\by_k\|^2.
\end{equation}
By multiplying both sides of \eqref{eq:convex_ineq_g} and \eqref{eq:g_smooth} by $A_{k}$ and $A_{k+1}$, respectively, and adding the resulted inequalities we obtain 
\begin{align*}
    &\phantom{{}={}}A_{k+1} (g(\bx_{k+1})-g(\by_k)) + A_k (g(\by_k)-g(\bx_k))\\
    &\leq A_{k+1} \fprod{\grad g(\by_k),\bx_{k+1}-\by_k}+ A_k \fprod{\grad g(\by_k),\by_k-\bx_k} + \frac{L_g A_{k+1}}{2}\|\bx_{k+1}-\by_k\|^2 \\
        &= a_k \fprod{\grad g(\by_k),\bz_{k+1}-\bz_k}+ A_k \fprod{\grad g(\by_k),\by_k-\bx_k} + \frac{L_g a^2_{k}}{2A_{k+1}}\|\bz_{k+1}-\bz_k\|^2 \\
    & = a_k \fprod{\grad g(\by_k),\bz_{k+1}-\by_k}+ \frac{L_g a^2_{k}}{2A_{k+1}}\|\bz_{k+1}-\bz_k\|^2,
\end{align*}
where the first equality holds since $a_k(\bz_{k+1}-\bz_k)=A_{k+1}(\bx_{k+1}-\by_k)$, and the second equality holds since  $a_k(\bz_k-\by_k)=A_k(\by_k-\bx_k)$. Lastly, by the definition of the constructed cutting plane, we know that $g(\by_k)+\fprod{\nabla g(\by_k),\bz-\by_k}\leq g_k$ for any $\bz$.  Hence,  $\fprod{\nabla g(\by_k),\bz_{k+1}-\by_k}$ is upper bounded by $g_k-g(\by_k)$ which itself is upper bounded by $g_k-g(\bx^*)$. Applying this substitution into to the above expression would lead to the claim in \eqref{eq:potential_g}.
\end{proof}

Now we are ready to prove Theorem~\ref{thm:upper_lower}.
\begin{proof}[Proof of Theorem~\ref{thm:upper_lower}]
To begin with, note that by our choice of $a_k$, we have 
\begin{equation}
    a_k = \frac{k+1}{4L_f} \quad \text{and} \quad A_{k+1} = \frac{(k+1)(k+2)}{8L_f}. 
\end{equation}
Thus, it can be verified that ${L_f a_k^2} \leq \frac{1}{2}A_{k+1}$. Then it follows from Lemma~\ref{lem:potential} that 
\begin{gather}
   A_{k+1}(f(\bx_{k+1})-f(\bx^*)) + \frac{1}{2}\|\bz_{k+1}-\bx^*\|^2 - \Bigl( A_{k}(f(\bx_{k})-f(\bx^*)) + \frac{1}{2}\|\bz_{k}-\bx^*\|^2\Bigr) \leq -\frac{1}{4}\|\bz_{k+1}-\bz_k\|^2 \label{eq:potential_f_1/4}\\
   A_{k+1}(g(\bx_{k+1})-g(\bx^*)) - A_{k}(g(\bx_{k})-g(\bx^*)) \leq a_k (g_k-g(\bx^*)) + \frac{L_g }{4L_f}\|\bz_{k+1}-\bz_k\|^2. \label{eq:potential_g_1/4}
\end{gather}

We first prove the convergence guarantee for the upper-level objective. By using induction on \eqref{eq:potential_f_1/4}, we obtain that for any $k\geq 0$
\begin{equation}
    A_{k}(f(\bx_{k})-f(\bx^*)) + \frac{1}{2}\|\bz_{k}-\bx^*\|^2 \leq A_{0}(f(\bx_{0})-f(\bx^*)) + \frac{1}{2}\|\bz_{0}-\bx^*\|^2 = \frac{1}{2}\|\bz_{0}-\bx^*\|^2,
\end{equation}
which implies 
\begin{equation*}
    f(\bx_{k})-f(\bx^*) \leq \frac{\|\bz_{0}-\bx^*\|^2}{2 A_k} = \frac{4L_f\|\bz_{0}-\bx^*\|^2}{k(k+1)}.
\end{equation*}

We proceed to establish an upper bound on $g(\bx_k)-g(\bx^*)$. 
By summing the inequality in \eqref{eq:potential_g_1/4} from $0$ to $k-1$ we obtain 
\begin{equation}\label{eq:summing_lower_level}
    \begin{aligned}
    A_k(g(\bx_k)-g(\bx^*)) &\leq \sum_{i=0}^{k-1} a_i(g_i-g^*) + \frac{L_g}{4L_f}\sum_{i=0}^{k-1} \|\bz_{i+1}-\bz_i\|^2 \\
    &\leq \sum_{i=0}^{k-1} \frac{i+1}{4L_f}\frac{2L_g\|\bx_0-\bx^*\|^2}{(i+1)^2} + \frac{L_g}{4L_f}\sum_{i=0}^{k-1} D^2 \\
    & = \frac{L_g}{2L_f}\|\bx_0-\bx^*\|^2 (\ln k+1) +\frac{L_g}{4L_f}D^2k.
\end{aligned}
\end{equation}
Note that the second inequality holds due to the condition in \eqref{eq:convergence_g}. Thus, we obtain
\begin{equation*}
    g(\bx_k)-g(\bx^*) \leq \frac{4L_g\|\bx_0-\bx^*\|^2 (\ln k+1)}{k(k+1)} +\frac{2L_gD^2}{k+1}.
\end{equation*}
The above upper bound on $ g(\bx_k)-g(\bx^*)$ without any additional condition, but next we show that if $f(\bx_k)\geq f(\bx^*)$ the above upper bound can be further improved as we can upper bound $\sum_{i=0}^{k-1} \|\bz_{i+1}-\bz_i\|^2$ by a constant independent of $k$ instead of $kD^2$.
To prove this claim, by summing the inequality in \eqref{eq:potential_f} from $0$ to $k-1$ we obtain
\begin{equation}
    \frac{1}{4}\sum_{i=0}^{k-1} \|\bz_{i+1}-\bz_i\|^2 \leq \frac{1}{2}\|\bz_0-\bx^*\|^2-\left(A_{k}(f(\bx_{k})-f(\bx^*)) + \frac{1}{2}\|\bz_{k}-\bx^*\|^2\right).
\end{equation}
Hence, if $f(\bx_k)\geq f(\bx^*)$, then it holds 
\begin{equation}
    \sum_{i=0}^{k-1} \|\bz_{i+1}-\bz_i\|^2 \leq 2\|\bz_0-\bx^*\|^2.
\end{equation}
Thus if replace $\sum_{i=0}^{k-1} \|\bz_{i+1}-\bz_i\|^2$ in \eqref{eq:summing_lower_level} by $2\|\bz_0-\bx^*\|^2$, we would obtain the following improve bound:
\begin{equation*}
    g(\bx_k)-g(\bx^*) \leq \frac{4L_g\|\bx_0-\bx^*\|^2(\ln k+1)}{k(k+1)}  +\frac{4L_g\|\bz_0-\bx^*\|^2}{k(k+1)}.
\end{equation*}
\end{proof}

Recall Remark~\ref{remark:why_compact}. The main difficulty in obtaining an accelerated rate of $\mathcal{O}(1/K^2)$ for $g$ is that it is unclear how to control $\sum_{k=0}^{K-1} \|\bz_{k+1}-\bz_k\|^2$. This, in turn, is because we don't know how to prove a \textbf{lower bound} on $f$. Instead, we used the compactness of $\cZ$ to achieve the $\mathcal{O}(1/K)$ for $g$.

\subsection{Proof of Lemma~\ref{lm:weighted_sum}}
\begin{proof}[Proof of Lemma~\ref{lm:weighted_sum}]

Note that by multiplying both sides of \eqref{eq:potential_f} by $\lambda>0$ we have
\begin{equation}\label{eq:potential_f_general}
\begin{aligned}
 A_{k+1}(\lambda (f(\bx_{k+1})-f(\bx^*))) + \frac{\lambda}{2}\|\bz_{k+1}-\bx^*\|^2 - \left( A_{k}(\lambda(f(\bx_{k})-f(\bx^*))) + \frac{\lambda}{2}\|\bz_{k}-\bx^*\|^2\right) \\
 \leq \frac{\lambda}{2}\left(\frac{L_f a_k^2}{A_{k+1}}-1\right)\|\bz_{k+1}-\bz_k\|^2
 \end{aligned}
 \end{equation}
 Further note that in this case we have $a_k = \frac{\gamma(k+1)}{4L_f}$ and $A_{k+1} = \gamma\frac{(k+1)(k+2)}{8L_f}$. Hence, $a_k^2/A_{k+1}$ is bounded above by $\frac{\gamma}{2L_f} $. 
 Therefore, we can replace $a_k^2/A_{k+1}$ in the above expression by $\frac{\gamma}{2L_f} $ to obtain 
\begin{equation}\label{eq:potential_f_gamma}
\begin{aligned}
 A_{k+1}(\lambda (f(\bx_{k+1})-f(\bx^*))) + \frac{\lambda}{2}\|\bz_{k+1}-\bx^*\|^2 - \left( A_{k}(\lambda(f(\bx_{k})-f(\bx^*))) + \frac{\lambda}{2}\|\bz_{k}-\bx^*\|^2\right) \\
 \leq \frac{\lambda}{2}\left(\frac{\gamma}{2}-1\right)\|\bz_{k+1}-\bz_k\|^2
 \end{aligned}
 \end{equation}
 Similarly, we can replace $a_k^2/A_{k+1}$ in \eqref{eq:potential_g} by $\frac{\gamma}{2L_f} $ to obtain
 \begin{equation}    
  A_{k+1}(g(\bx_{k+1})-g(\bx^*)) - A_{k}(g(\bx_{k})-g(\bx^*)) \leq a_k (g_k-g(\bx^*)) + \frac{\gamma L_g }{4L_f}\|\bz_{k+1}-\bz_k\|^2
 \end{equation}
 Note if we add above up the above inequalities in \eqref{eq:potential_f_general} and \eqref{eq:potential_f_gamma} we obtain
\begin{equation}
\begin{aligned}
   A_{k+1}(\lambda(f(\bx_{k+1})-f(\bx^*)) + g(\bx_{k+1})-g(\bx^*)) + \frac{\lambda}{2}\|\bz_{k+1}-\bx^*\|^2 \\
   -\left( A_{k}(\lambda(f(\bx_{k})-f(\bx^*)) + g(\bx_{k})-g(\bx^*)) + \frac{\lambda}{2}\|\bz_{k}-\bx^*\|^2\right) \\
   \leq (\lambda(\frac{\gamma}{4}-\frac{1}{2})+\frac{\gamma}{4}\frac{L_{g}}{L_{f}})\|\bz_{k+1}-\bz_k\|^2 + a_k (g_k-g(\bx^*)) 
   \leq a_k (g_k-g(\bx^*)) 
\end{aligned}
\end{equation}
Note that the last inequality holds since the first term is negative due to the choice of $\lambda$.
By summing the inequalities from $ 0$ to $k-1$ we obtain
\begin{equation}\label{eq:jjjjj}
    A_{k}(\lambda(f(\bx_{k})-f(\bx^*)) + g(\bx_{k})-g(\bx^*)) + \frac{1}{2}\lambda\|\bz_{k}-\bx^*\|^2  
   \leq\sum_{i=0}^{k-1}a_i (g_i-g(\bx^*)) + \frac{1}{2}\lambda\|\bz_{0}-\bx^*\|^2
\end{equation}
Now using the condition on $g_i$ in \eqref{eq:convergence_g} and the definition of $a_i$ we have 
\begin{equation}
    \sum_{i=0}^{k-1}a_i (g_i-g(\bx^*))\leq \sum_{i=0}^{k-1} \frac{\gamma(i+1)}{4L_f}\frac{2L_g\|\bx_0-\bx^*\|^2}{(i+1)^2} \leq \frac{\gamma L_g}{2L_f}\|\bx_0-\bx^*\|^2 (\ln k+1)
\end{equation}
By applying this upper bound into \eqref{eq:jjjjj} we obtain
\begin{equation}
   A_{k}(\lambda(f(\bx_{k})-f(\bx^*)) + g(\bx_{k})-g(\bx^*)) + \frac{1}{2}\lambda\|\bz_{k}-\bx^*\|^2 
   \leq \frac{\gamma L_g}{2L_f}\|\bx_0-\bx^*\|^2 (\ln k+1) + \frac{1}{2}\lambda\|\bz_{0}-\bx^*\|^2 
\end{equation}
If we drop the $\frac{1}{2}\lambda\|\bz_{k}-\bx^*\|^2$ in the left-hand side and divide both sides of the resulted inequality by $A_{K}$ which is equal to $A_k=\gamma\frac{k(k+1)}{8L_f}$ we obtain
\begin{equation}
   \lambda(f(\bx_{k})-f(\bx^*)) + g(\bx_{k})-g(\bx^*) \leq \frac{4L_g\|\bx_0-\bx^*\|^2 (\ln k+1)}{k(k+1)}  + \frac{4\lambda L_{f}\|\bz_{0}-\bx^*\|^2}{\gamma k(k+1)}, 
\end{equation}
and the proof is complete. 
\end{proof}


\subsection{Proof of Theorem~\ref{thm:upper_lower_hd}}
\begin{proof}[Proof of Theorem~\ref{thm:upper_lower_hd}]

Recall the result of Lemma~\ref{lm:weighted_sum} that  if $a_k = \gamma(k+1)/(4L_f)$, where $0<\gamma\leq1$ and $\lambda \geq \frac{L_g}{(2/\gamma-1)L_f}$  then after $T$ iterations we have 
\begin{equation}
    \lambda(f(\bx_{T})-f(\bx^*)) + g(\bx_{T})-g(\bx^*) 
       \leq \frac{4L_g\|\bx_0-\bx^*\|^2(\ln T+1)}{T(T+1)}   + \frac{4\lambda L_{f}\|\bz_{0}-\bx^*\|^2}{\gamma T(T+1)}.
\end{equation}
Now if we replace $\gamma$ by $ 1/(\frac{2L_g}{L_f}T^{\frac{2r-2}{2r-1}}+2)$ as suggested in the statement of the theorem,  we would obtain
\begin{equation}\label{eq:comb}
    \lambda(f(\bx_{T})-f(\bx^*)) + g(\bx_{T})-g(\bx^*) 
       \leq \frac{4L_g\|\bx_0-\bx^*\|^2(\ln T+1)}{T(T+1)}   + \frac{8 (\frac{L_g}{L_f}T^{\frac{2r-2}{2r-1}}+1)\lambda L_{f}\|\bz_{0}-\bx^*\|^2}{ T(T+1)}.
\end{equation}
Now we proceed to prove the first claim which is an upper bound on $f(\bx_{T})-f(\bx^*)$. Note that given the fact that $g(\bx_{T})-g(\bx^*) >0$ and $\lambda>0$ we can show that  
\begin{equation}
  f(\bx_{T})-f(\bx^*)
       \leq \frac{4L_g\|\bx_0-\bx^*\|^2(\ln T+1)}{\lambda T(T+1)}   + \frac{8( (\frac{L_g}{L_f}T^{\frac{2r-2}{2r-1}}+1)) L_{f}\|\bz_{0}-\bx^*\|^2}{ T(T+1)}.
\end{equation}
If we select $\lambda$ which is a free parameter as $\lambda = T^{-\frac{2r-2}{2r-1}} \geq \frac{L_g}{(2/\gamma-1)L_f}$ then we obtain
\begin{equation}
  f(\bx_{T})-f(\bx^*)
       \leq \frac{4L_g\|\bx_0-\bx^*\|^2(\ln T+1)}{T^{\frac{1}{2r-1}} (T+1)}  + \frac{8L_g \|\bz_{0}-\bx^*\|^2}{ T^{\frac{1}{2r-1}} (T+1)} + \frac{8 L_{f}\|\bz_{0}-\bx^*\|^2}{ T(T+1)}.
\end{equation}
Given the fact that $\bx_0=\bz_0$ we can simplify the upper bound to
\begin{equation}
  f(\bx_{T})-f(\bx^*)
       \leq \frac{12L_g\|\bx_0-\bx^*\|^2(\ln T+1)}{T^{\frac{2r}{2r-1}} }   + \frac{8 L_{f}\|\bz_{0}-\bx^*\|^2}{ T^2}.
\end{equation}

Next, we proceed to establish an upper bound on $g(\bx_{T})-g(\bx^*) $. We will use the following inequality that holds due to the HEB condition and formally stated in Proposition~\ref{pp:holderian}:
\begin{equation}
    f(\mathbf{x}_{T})-f^* \geq -M\left(\frac{r (g(\bx_{T})-g(\bx^*)))}{\alpha}\right)^{\frac{1}{r}}
\end{equation}
Now if we replace this lower bound into \eqref{eq:comb} we would obtain
\begin{equation}\label{eq:comb2}
\begin{aligned}
   &-\lambda M(\frac{r}{\alpha})^{\frac{1}{r}}(g(\bx_{T})-g(\bx^*))^{\frac{1}{r}} + g(\bx_{T})-g(\bx^*) \\
   &\leq \frac{4L_g\|\bx_0-\bx^*\|^2(\ln T+1)}{T(T+1)}   + \frac{8( (\frac{L_g}{L_f}T^{\frac{2r-2}{2r-1}}+1))\lambda L_{f}\|\bz_{0}-\bx^*\|^2}{ T(T+1)}.
   \end{aligned}
\end{equation}
Next we consider two different cases: In the first case we assume $\lambda M(\frac{r}{\alpha})^{\frac{1}{r}}(g(\bx_{T})-g(\bx^*))^{\frac{1}{r}} \leq \frac{1}{2}(g(\bx_{T})-g(\bx^*))$ holds and in the second case we assume the opposite of this inequality holds. 

If we are in the first case and $\lambda M(\frac{r}{\alpha})^{\frac{1}{r}}(g(\bx_{T})-g(\bx^*))^{\frac{1}{r}} \leq \frac{1}{2}(g(\bx_{T})-g(\bx^*))$, then the inequality in \eqref{eq:comb2} leads to 
\begin{equation}\label{eq:comb3}
  \frac{1}{2}( g(\bx_{T})-g(\bx^*)) \leq \frac{4L_g\|\bx_0-\bx^*\|^2(\ln T+1)}{T(T+1)}   + \frac{8( (\frac{L_g}{L_f}T^{\frac{2r-2}{2r-1}}+1))\lambda L_{f}\|\bz_{0}-\bx^*\|^2}{ T(T+1)}.
\end{equation}
Since $\lambda = T^{-\frac{2r-2}{2r-1}}$, it further leads to the following upper bound
\begin{equation}\label{eq:comb4}
 g(\bx_{T})-g(\bx^*) \leq \frac{8L_g\|\bx_0-\bx^*\|^2(\ln T+1)}{T(T+1)}   
 + \frac{16{L_g} \|\bz_{0}-\bx^*\|^2}{ T(T+1)}
  + \frac{16  L_{f}\|\bz_{0}-\bx^*\|^2}{ T^\frac{1}{2r-1}(T+1)}.
\end{equation}
Now given the fact that $\bx_0=\bz_0$, we obtain
\begin{equation}\label{eq:comb5}
 g(\bx_{T})-g(\bx^*) \leq \frac{24L_g\|\bx_0-\bx^*\|^2(\ln T+1)}{T^2}   
  + \frac{16  L_{f}\|\bx_{0}-\bx^*\|^2}{ T^\frac{2r}{2r-1}}.
\end{equation}

If we are in the second case and $\lambda M(\frac{r}{\alpha})^{\frac{1}{r}}(g(\bx_{T})-g(\bx^*))^{\frac{1}{r}} > \frac{1}{2}(g(\bx_{T})-g(\bx^*))$ then this inequality is equivalent to 
\begin{equation}
    (g(\bx_{T})-g(\bx^*))^{1-{1/r}}\leq  2 \lambda M(\frac{r}{\alpha})^{\frac{1}{r}},
\end{equation}
leading to 
\begin{equation}\label{eq:comb6}
    g(\bx_{T})-g(\bx^*)\leq \left( 2 T^{-\frac{2r-2}{2r-1}} M(\frac{r}{\alpha})^{\frac{1}{r}}\right)^{\frac{r}{r-1}}    =\frac{(2M)^{\frac{r}{r-1} }  (\frac{r}{\alpha})^{\frac{1}{r-1}}}{T^{\frac{2r}{2r-1}}}
\end{equation}
By combining the bounds in \eqref{eq:comb5} and \eqref{eq:comb6} we realize that 
\begin{equation}\label{eq:comb7}
    g(\bx_{T})-g(\bx^*)\leq \max\{\frac{24L_g\|\bx_0-\bx^*\|^2(\ln T+1)}{T^2}+ \frac{16  L_{f}\|\bx_{0}-\bx^*\|^2}{T^{\frac{2r}{2r-1}}}, \frac{(2M)^{\frac{r}{r-1} }  (\frac{r}{\alpha})^{\frac{1}{r-1}}}{T^{\frac{2r}{2r-1}}}\}
\end{equation}

Finally by using the above bound in \eqref{eq:comb7} and the result of Proposition~\ref{pp:holderian} we can prove the the second claim and establish a lower bound on $f(\bx_{T})-f^*$ which is 
\begin{equation}
    f(\bx_{T})-f^* \geq -M\left(\frac{r}{\alpha}\right)^{\frac{1}{r}}
    \left(\max\{\frac{24L_g\|\bx_0-\bx^*\|^2(\ln T+1)}{T^2}+ \frac{16  L_{f}\|\bx_{0}-\bx^*\|^2}{T^{\frac{2r}{2r-1}}}, \frac{(2M)^{\frac{r}{r-1} }  (\frac{r}{\alpha})^{\frac{1}{r-1}}}{T^{\frac{2r}{2r-1}}}\}\right)^{1/r}
\end{equation}
leading to 
\begin{equation}
\begin{aligned}
    &f({\mathbf{x}}_T)-f^* \geq \\
    &-M\left(\frac{r}{\alpha}\right)^{\frac{1}{r}}
    \left(\max\{\frac{(24L_g\|\bx_0-\bx^*\|^2(\ln T+1))^{1/r}}{T^{2/r}}+ \frac{(16  L_{f}\|\bx_{0}-\bx^*\|^2)^{1/r}}{T^{\frac{2}{2r-1}}}, \frac{((2M)^{\frac{r}{r-1} }  (\frac{r}{\alpha})^{\frac{1}{r-1}})^{1/r}}{T^{\frac{2}{2r-1}}}\}\right)
\end{aligned}
\end{equation}
\end{proof}

\subsection{Proof of Theorem~\ref{pp:upper_lower_ws}}
\begin{proof}[Proof of Theorem~\ref{pp:upper_lower_ws}]
    To upper bound $f(\bx_k)-f(\bx^*)$, we follow a similar analysis as in Theorem~\ref{thm:upper_lower}. Specifically,
    first note that by our choice of $a_k$, we have 
    $a_k = \gamma \frac{k+1}{4L_f}$ and $A_{k+1} = \gamma\frac{(k+1)(k+2)}{8L_f}$, where $\gamma \in (0,1)$. 
Hence, we can obtain that ${L_f a_k^2} \leq \frac{\gamma}{2}A_{k+1}$. By using Lemma~\ref{lem:potential} and the fact that $\gamma \in (0,1)$, we have 
\begin{equation*}
\begin{aligned}
    A_{k+1}(f(\bx_{k+1})-f(\bx^*)) + \frac{1}{2}\|\bz_{k+1}-\bx^*\|^2 - \Bigl( A_{k}(f(\bx_{k})-f(\bx^*)) + \frac{1}{2}\|\bz_{k}-\bx^*\|^2\Bigr) \\
    \leq \left(\frac{\gamma}{4}-\frac{1}{2}\right)\|\bz_{k+1}-\bz_k\|^2 \leq 0.
\end{aligned}
\end{equation*}
By using induction, we obtain that for any $k\geq 0$ 
\begin{equation}
    A_{k}(f(\bx_{k})-f(\bx^*)) + \frac{1}{2}\|\bz_{k}-\bx^*\|^2 \leq A_{0}(f(\bx_{0})-f(\bx^*)) + \frac{1}{2}\|\bz_{0}-\bx^*\|^2  = \frac{1}{2}\|\bz_{0}-\bx^*\|^2,
\end{equation}
Since $A_k = \gamma\frac{k(k+1)}{8L_f}$ and $\bz_0 = \bx_0$, this further implies that 
    \begin{equation*}
    f(\bx_{k})-f(\bx^*) \leq \frac{\|\bz_{0}-\bx^*\|^2}{2 A_k} = \frac{4L_f\|\bx_{0}-\bx^*\|^2}{\gamma k(k+1)}.
    \end{equation*}
    
    Next, we will prove the upper bound on $g(\bx_k)-g(\bx^*)$. By Lemma~\ref{lm:weighted_sum}, we have for any $k\geq 0$
    \begin{equation}\label{eq:weak_sharp_weighted}
       \lambda(f(\bx_{k})-f(\bx^*)) + g(\bx_{k})-g(\bx^*) \leq \frac{4L_g\|\bx_0-\bx^*\|^2}{k(k+1)} (\ln k+1)  + \frac{4\lambda L_{f}\|\bx_{0}-\bx^*\|^2}{\gamma k(k+1)} 
    \end{equation}
    Moreover, since $g$ satisfies the weak sharpness condition, we can use Proposition~\ref{pp:holderian} with $r = 1$ to write 
    \begin{equation}\label{eq:weak_sharp_bound}
        f(\bx_{k})-f^* \geq -\frac{M}{\alpha}(g(\bx_{k})-g^*).
    \end{equation}
    Combining \eqref{eq:weak_sharp_weighted} and \eqref{eq:weak_sharp_bound} leads to
\begin{equation}\label{eq:before_setting_lambda}
       -\lambda\frac{M}{\alpha}(g(\bx_{k})-g(\bx^*)) + g(\bx_{k})-g(\bx^*) \leq \frac{4L_g\|\bx_0-\bx^*\|^2}{k(k+1)} (\ln k+1)  + \frac{4\lambda L_{f}\|\bx_{0}-\bx^*\|^2}{\gamma k(k+1)} 
    \end{equation}
    Note that we can choose $\lambda$ to be any number satisfying $\lambda \geq \frac{L_g}{(2/\gamma-1)L_f}$ (cf. Lemma~\ref{lm:weighted_sum}). 
    Specifically, since $\gamma \leq \frac{2\alpha L_f}{2ML_g+\alpha L_f}$, we can set $\lambda = \alpha/(2M)$ and accordingly \eqref{eq:before_setting_lambda} can be simplified to  
    \begin{equation*}
        g(\bx_{k})-g(\bx^*) \leq \frac{8L_g\|\bx_0-\bx^*\|^2}{k(k+1)} (\ln k+1)  + \frac{4\alpha L_{f}\|\bx_{0}-\bx^*\|^2}{\gamma Mk(k+1)}.
    \end{equation*}
    Finally, we use \eqref{eq:weak_sharp_bound} again together with the above upper bound on $ g(\bx_{k})-g(\bx^*)$ to obtain 
    \begin{equation*}
    f(\bx_k) - f(\bx^*) \geq -\frac{M}{\alpha}(g(\bx_k)-g(\bx^*)) \geq -\left(\frac{8ML_g\|\bx_0-\bx^*\|^2}{\alpha k(k+1)} (\ln k+1) +\frac{4L_f\|\bx_{0}-\bx^{*}\|^2}{\gamma k(k+1)}\right).
    \end{equation*}
\end{proof}

\section{Extension to the Non-smooth/Composite Setting}\label{sec:extension}
In this section, we would like to mention the possible extension to the non-smooth/composite setting. In the general non-smooth settings, we believe it is not possible to extend our results and achieve the purpose of the acceleration. This is because, even in the single-level
setting, the best achievable rate in the general non-smooth setting is $\mathcal{O}(1/\sqrt{K})$ achieved by sub-gradient method. That said, it should be possible to extend our accelerated bilevel framework to a special non-smooth setting where the upper- and lower-level objective functions have a composite structure, i.e., they can be written as the sum of a convex smooth function and a convex non-smooth function that is easy to compute its proximal operator.

Then we consider the composite counterpart of Problem~\eqref{eq:bi-simp}:
\begin{equation}\label{eq:prox-bi-simp}
    \min_{\bx\in \reals^n}~f(\bx):= f_{1}(\bx) + f_{2}(\bx)\qquad \hbox{s.t.}\quad  \bx\in\argmin_{\bz\in \reals^n}~g(\bz):= g_{1}(\bz)+g_{2}(\bz),
\end{equation} 
where $f_1,g_1:\reals^n\to \reals$ are smooth convex functions and $f_2,g_2:\reals^n\to \reals$ are nonsmooth convex functions, respectively. 

To analyze and implement the proximal gradient-based methods, we need the following definition concerning the property of the proximal mapping.

\begin{definition}\label{def:prox}
    Given a function $h : \reals^{n} \to (-\infty, +\infty]$, the proximal map of $h$ is defined for all $\bx \in \reals^{n}$ and $\eta > 0$ as
    \begin{equation}
        \textit{Prox}_{\eta h}(\bx) \triangleq \argmin_{\bu \in \reals^{n}}\{\frac{1}{2\eta}\|\bu-\bx\|^{2}+h(\bu)\}
    \end{equation}
\end{definition}

To handle the upper-level nonsmooth part $f_{2}$, we need to change the projection step in Step 6 of Algorithm~\ref{alg:AGM-BiO} to a proximal update (outlined in the Step 6 of Algorithm~\ref{alg:P-AGM-BiO}), which is similar to the accelerated proximal gradient method for single-level problems in \cite{beck2009fast}. 
On the other hand, to deal with the lower-level nonsmooth part $g_{2}$, it is necessary to modify the approximated set of the lower-level solution set $\mathcal{X}_{k}$. Specifically, since $g_{2}$ is nonsmooth,  $\nabla g$ may not be obtained over the whole domain. Hence, we utilize $g_{1}$ and $g_{2}$ separately in the construction of $\mathcal{X}_{k}$ by summing a linear approximation of $g_{1}$ and $g_{2}$ up as a lower bound of $g_{k}$. Note that the constructed set $\mathcal{X}_{k}$ is no longer a hyperplane in this setting due to the possibly non-linear nature of $g_{2}$. We refer to our method as the Proximal Accelerated Gradient Method for Bilevel Optimization (P-AGM-BiO) and its steps are outlined in Algorithm~\ref{alg:P-AGM-BiO}.

\begin{algorithm}[t!]
	\caption{Proximal Accelerated Gradient Method for Bilevel Optimization (P-AGM-BiO)}\label{alg:P-AGM-BiO}
	\begin{algorithmic}[1]
	\STATE \textbf{Input}: A sequence $\{g_k\}_{k=0}^K$, a scalar $\gamma \in (0,1]$
	\STATE \textbf{Initialization}: $A_0 = 0$, $\bx_0=\bz_0\in \reals^n$
	\FOR{$k = 0,\dots,K$}
	    \STATE Set  $\quad \displaystyle{a_{k} = \gamma\frac{k+1}{4L_{f}}}$
     \STATE Compute 
     $\quad \displaystyle{\by_k = \frac{A_k}{A_k+a_k} \bx_k + \frac{a_k}{A_k+a_k} \bz_k}$

     \vspace{1mm}
     \STATE Compute 
 $\quad \displaystyle{
        \bz_{k+1} = Prox_{a_{k}(f_{2}+\delta_{{\mathcal{X}_k}})}(\bz_k-a_k \nabla f_{1}(\by_k))}$, where $$ \displaystyle{ \mathcal{X}_k \triangleq \{\bz\in \reals^{n}: g_{1}(\by_k)+\fprod{\nabla g_{1}(\by_k),\bz-\by_k} + g_{2}(\bz)}\leq g_k\}$$

      \vspace{-2mm}
     \STATE Compute    $\quad \displaystyle{ \bx_{k+1} = \frac{A_k}{A_k+a_k} \bx_k+\frac{a_k}{A_k+a_k} \bz_{k+1}}$ 
  \vspace{1mm}
		\STATE Update $A_{k+1} = A_k + a_k$
	\ENDFOR
 
    \STATE \textbf{Return:} $\bx_{K}$ 
	\end{algorithmic}
\end{algorithm}





Different proximal-friendly assumptions are commonly used in the literature of composite single-level/bilevel optimization \cite{beck2009fast, samadi2023achieving, wang2024near, chen2024penalty}. The following proximal-friendly assumption is necessary for our method in the composite setting. 
\begin{assumption}\label{ass:additional}
    The function $f_{2} + \delta_{\mathcal{X}_{k}}$ in the Step 6 of Algorithm~\ref{alg:P-AGM-BiO} is proximal-friendly, i.e. the proximal mapping in Definition~\ref{def:prox} is easy to compute, for all $k$, where $\delta_{\mathcal{X}_{k}}(\cdot)$ is the indicator function.
\end{assumption}

This assumption 
implies that 
$f_{2}$ is proximal-friendly and projecting onto the constructed set $\mathcal{X}_{k}$ is easy. Moreover, The function $f_{2} + \delta_{\mathcal{X}_{k}}$ is the sum of two convex functions, and the study of proximal mapping for sums of functions can be found in the literature \cite{yu2013decomposing, pustelnik2017proximity, bauschke2018projecting, adly2019decomposition}.

Note that the properties of smoothness of $f$ and $g$ have only been used in the proof of Lemma~\ref{lem:potential} in Section~\ref{sec:analysis}. None of the other results will break if \eqref{eq:potential_f} and \eqref{eq:potential_g} in Lemma~\ref{lem:potential} still hold in the composite setting. Now, we present and prove the counterpart of  Lemma~\ref{lem:potential} in the composite setting.

\begin{lemma}\label{lem:potential_prox}
Suppose $f_1, f_2, g_1, g_2$ are convex and $f_1, g_1$ are $L_{f}$-smooth and $L_{g}$-smooth, respectively. Let $\{\bx_k\}$ be the sequence of iterates generated by Algorithm~\ref{alg:P-AGM-BiO} with stepsize $a_k>0$ for
$k \geq 0$. Moreover, suppose Assumption~\ref{ass:additional} holds. Then we have 
\begin{equation}
    \begin{aligned}
        A_{k+1}(f(\bx_{k+1})-f(\bx^*)) + \frac{1}{2}\|\bz_{k+1}-\bx^*\|^2 -& \Bigl( A_{k}(f(\bx_{k})-f(\bx^*)) + \frac{1}{2}\|\bz_{k}-\bx^*\|^2\Bigr) \\
        &\leq \left(\frac{L_f a_k^2}{2A_{k+1}}-\frac{1}{2}\right)\|\bz_{k+1}-\bz_k\|^2, \\
    \end{aligned}
\end{equation}
\begin{equation}
     A_{k+1}(g(\bx_{k+1})-g(\bx^*)) - A_{k}(g(\bx_{k})-g(\bx^*)) \leq a_k (g_k-g(\bx^*)) + \frac{L_g a^2_{k}}{2A_{k+1}}\|\bz_{k+1}-\bz_k\|^2.   
\end{equation}

\end{lemma}

\begin{proof}[Proof of Lemma~\ref{lem:potential_prox}]
Let $\bx^*$ be any optimal solution of \eqref{eq:prox-bi-simp}. 

We first consider the upper-level objective $f$. Since $f_1$ is convex, we have 
\begin{equation}\label{eq:convex_ineq_prox}
    f_1(\by_k) - f_1(\bx^*) \leq \fprod{\grad f_1(\by_k),\by_k-\bx^*},\quad 
    f_1(\by_k) - f_1(\bx_k) \leq \fprod{\grad f_1(\by_k),\by_k-\bx_k}.
\end{equation}
Now given the update rule $A_{k+1}=A_k+a_k$, we can write
\begin{equation}\label{eq:weighted_inequality_prox}
    A_{k+1}(f_1(\by_k)-f_1(\bx^*)) - A_k(f_1(\bx_k)-f_1(\bx^*)) = a_k(f_1(\by_k) - f_1(\bx^*))+A_k(f_1(\by_k) - f_1(\bx_k))
\end{equation}
Combining \eqref{eq:convex_ineq_prox} and \eqref{eq:weighted_inequality_prox}, we have 
\begin{equation}\label{eq:anytime_online_batch_prox}
    \begin{aligned}
    A_{k+1}(f_1(\by_k)-f_1(\bx^*)) - A_k(f_1(\bx_k)-f_1(\bx^*)) 
    &\leq  a_k(\fprod{\grad f_1(\by_k),\by_k-\bx^*})+A_k(\fprod{\grad f_1(\by_k),\by_k-\bx_k}) \\
    &= \fprod{\grad f_1(\by_k), a_k \by_k +A_k(\by_k-\bx_k)-a_k\bx^*} \\
    &= a_k \fprod{\grad f_1(\by_k),\bz_k-\bx^*},
\end{aligned}
\end{equation}
where the last equality follows from the definition of $\by_k$. Furthermore, since $f_1$ is $L_f$-smooth, we have 
\begin{equation}\label{eq:f_smooth_prox}
    f_1(\bx_{k+1}) \leq f_1(\by_k)+\fprod{\grad f_1(\by_k),\bx_{k+1}-\by_k}+\frac{L_f}{2}\|\bx_{k+1}-\by_k\|^2.
\end{equation}
If we multiply both sides of \eqref{eq:f_smooth_prox} by $A_{k+1}$ and combine the resulting inequality with \eqref{eq:anytime_online_batch_prox}, we obtain
\begin{equation}\label{eq:potential_prox}
    \begin{aligned}
    &\phantom{{}\leq{}} A_{k+1}(f_1(\bx_{k+1})-f_1(\bx^*)) - A_k(f_1(\bx_k)-f_1(\bx^*))\\
    & \leq a_k \fprod{\grad f_1(\by_k),\bz_k-\bx^*}+ A_{k+1} \fprod{\grad f_1(\by_k),\bx_{k+1}-\by_k}+\frac{L_f A_{k+1}}{2}\|\bx_{k+1}-\by_k\|^2 \\
    & = a_k \fprod{\grad f_1(\by_k),\bz_k-\bx^*}+ a_{k} \fprod{\grad f_1(\by_k),\bz_{k+1}-\bz_k}+\frac{L_fa_k^2}{2A_{k+1}}\|\bz_{k+1}-\bz_k\|^2 \\
    & = a_k \fprod{\grad f_1(\by_k),\bz_{k+1}-\bx^*}+\frac{L_fa_k^2}{2A_{k+1}}\|\bz_{k+1}-\bz_k\|^2,
\end{aligned}
\end{equation}
where we used the fact that $a_k(\bz_{k+1}-\bz_k)=A_{k+1}(\bx_{k+1}-\by_k)$ in the first equality. 
Moreover, from the step 6 in Algorithm~\ref{alg:P-AGM-BiO}, we have $\bz_{k} - a_k\nabla f_1(\by_{k}) - \bz_{k+1} \in a_{k}\partial(f_{2}(\bz_{k+1})+\delta_{\mathcal{X}_{k}}(\bz_{k+1}))$. Using this, from the definition of subgradients for $f_{2}+\delta_{\mathcal{X}_k}$, we have
\begin{equation}
   \begin{aligned}\label{eq:prox}
   &\fprod{\bx^{*}-\bz_{k+1}, \bz_{k} - a_k\nabla f_1(\by_{k}) - \bz_{k+1}} \leq a_{k}(f_{2}(\bx^*)+\delta_{\mathcal{X}_{k}}(\bx^*)-f_{2}(\bz_{k+1})-\delta_{\mathcal{X}_{k}}(\bz_{k+1}))\\
    \Leftrightarrow \quad &\fprod{\bz_{k+1}-\bz_k+a_k\nabla f(\by_k),\bx^*-\bz_{k+1}} \geq a_kf_2(\bz_{k+1})-a_kf_2(\bx^*)\\
    \Leftrightarrow \quad & a_k\fprod{\nabla f(\by_k),\bz_{k+1}-\bx^*} \leq \fprod{\bz_{k+1}-\bz_k,\bx^*-\bz_{k+1}} - a_kf_2(\bz_{k+1})+a_kf_2(\bx^*)\\
    \Leftrightarrow \quad & a_k\fprod{\nabla f(\by_k),\bz_{k+1}-\bx^*} \leq \frac{1}{2}\|\bz_k-\bx^*\|^2-\frac{1}{2}\|\bz_{k+1}-\bx^*\|^2-\frac{1}{2}\|\bz_{k+1}-\bz_k\|^2\\
    &\qquad \qquad \qquad \qquad \qquad \quad-a_kf_2(\bz_{k+1})+a_kf_2(\bx^*).
\end{aligned} 
\end{equation}
The first step holds since $\bx^*, \bz_{k+1} \in \mathcal{X}_{k}$, i.e. $\delta_{\mathcal{X}_{k}}(\bx^*) = \delta_{\mathcal{X}_{k}}(\bz_{k+1}) = 0$.
Combining \eqref{eq:potential_prox} and \eqref{eq:prox} leads to 
\begin{equation}
\begin{aligned}
    &A_{k+1}(f_1(\bx_{k+1})-f_1(\bx^*)) + \frac{1}{2}\|\bz_{k+1}-\bx^*\|^2 - \left( A_{k}(f_1(\bx_{k})-f_1(\bx^*)) + \frac{1}{2}\|\bz_{k}-\bx^*\|^2\right) \\
    &\leq \frac{1}{2}\left(\frac{L_f a_k^2}{A_{k+1}}-1\right)\|\bz_{k+1}-\bz_k\|^2 -a_kf_2(\bz_{k+1})+a_kf_2(\bx^*),
\end{aligned}
\end{equation}
Then we add $(A_{k+1}f_2(\bx_{k+1}) - a_kf_{2}(\bx^*) -A_kf_{2}(\bx_{k}))$ on both sides to obtain, 
\begin{equation}
\begin{aligned}
    &A_{k+1}(f(\bx_{k+1})-f(\bx^*)) + \frac{1}{2}\|\bz_{k+1}-\bx^*\|^2 - \left( A_{k}(f(\bx_{k})-f(\bx^*)) + \frac{1}{2}\|\bz_{k}-\bx^*\|^2\right) \\
    &\leq \frac{1}{2}\left(\frac{L_f a_k^2}{A_{k+1}}-1\right)\|\bz_{k+1}-\bz_k\|^2 -a_kf_2(\bz_{k+1})+ A_{k+1}f_2(\bx_{k+1}) -A_kf_{2}(\bx_{k}),
\end{aligned}
\end{equation}
Finally, by the convexity of $f_2$, $A_{k+1} = A_{k}+a_{k}$, and $\bx_{k+1} = \frac{A_k}{A_k+a_k} \bx_k+\frac{a_k}{A_k+a_k} \bz_{k+1}$, i.e. $-a_kf_2(\bz_{k+1})+ A_{k+1}f_2(\bx_{k+1}) -A_kf_{2}(\bx_{k}) \leq 0$, the first inequality of this Lemma can be obtained.

Next, we proceed to prove the claim for the lower-level objective $g$. To do so, we first leverage the convexity of the smooth part $g_1$ which leads to
\begin{equation}\label{eq:convex_ineq_g1}
    g_1(\by_k) - g_1(\bx_k) \leq \fprod{\grad g_1(\by_k),\by_k-\bx_k}.
\end{equation}
Also, since $g_1$ is $L_g$-smooth, we have 
\begin{equation}\label{eq:g1_smooth}
    g_1(\bx_{k+1}) \leq g_1(\by_k)+\fprod{\grad g_1(\by_k),\bx_{k+1}-\by_k}+\frac{L_g}{2}\|\bx_{k+1}-\by_k\|^2.
\end{equation}
By multiplying both sides of \eqref{eq:convex_ineq_g1} and \eqref{eq:g1_smooth} by $A_{k}$ and $A_{k+1}$, respectively, and adding the resulted inequalities we obtain 
\begin{align*}
    &\phantom{{}={}}A_{k+1} (g_1(\bx_{k+1})-g_1(\by_k)) + A_k (g_1(\by_k)-g_1(\bx_k))\\
    &\leq A_{k+1} \fprod{\grad g_1(\by_k),\bx_{k+1}-\by_k}+ A_k \fprod{\grad g_1(\by_k),\by_k-\bx_k} + \frac{L_g A_{k+1}}{2}\|\bx_{k+1}-\by_k\|^2 \\
        &= a_k \fprod{\grad g_1(\by_k),\bz_{k+1}-\bz_k}+ A_k \fprod{\grad g_1(\by_k),\by_k-\bx_k} + \frac{L_g a^2_{k}}{2A_{k+1}}\|\bz_{k+1}-\bz_k\|^2 \\
    & = a_k \fprod{\grad g_1(\by_k),\bz_{k+1}-\by_k}+ \frac{L_g a^2_{k}}{2A_{k+1}}\|\bz_{k+1}-\bz_k\|^2,
\end{align*}
where the first equality holds since $a_k(\bz_{k+1}-\bz_k)=A_{k+1}(\bx_{k+1}-\by_k)$, and the second equality holds since  $a_k(\bz_k-\by_k)=A_k(\by_k-\bx_k)$. Lastly, by the definition of the constructed approximated set $\mathcal{X}_{k}$, we know that $g_1(\by_k)+\fprod{\nabla g_1(\by_k),\bz-\by_k} + g_2(\bz)\leq g_k$ for any $\bz \in \mathcal{X}_{k}$.  Hence,  $\fprod{\nabla g_1(\by_k),\bz_{k+1}-\by_k}$ is upper bounded by $g_k-g_1(\by_k)-g_2(\bz_{k+1})$. Applying this substitution into to the above expression to obtain, 
\begin{equation}
    \begin{aligned}
        &\phantom{{}={}}A_{k+1} (g_1(\bx_{k+1})-g_1(\by_k)) + A_k (g_1(\by_k)-g_1(\bx_k))\\
        & \leq a_kg_k - a_kg_1(\by_k) - a_kg_2(\bz_{k+1}) + \frac{L_g a^2_{k}}{2A_{k+1}}\|\bz_{k+1}-\bz_k\|^2
    \end{aligned}
\end{equation}
By adding $a_kg_1(\by_k) - a_kg_1(\bx^*)$ on both sides, we have, 
\begin{equation}
    \begin{aligned}
        &\phantom{{}={}}A_{k+1} (g_1(\bx_{k+1})-g_1(\bx^*)) + A_k (g_1(\bx^*)-g_1(\bx_k))\\
        & \leq a_kg_k - a_kg_1(\bx^*) - a_kg_2(\bz_{k+1}) + \frac{L_g a^2_{k}}{2A_{k+1}}\|\bz_{k+1}-\bz_k\|^2
    \end{aligned}
\end{equation}
Lastly, we add $(A_{k+1}g_2(\bx_{k+1}) - a_kg_{2}(\bx^*) -A_kg_{2}(\bx_{k}))$ on both sides to obtain,
\begin{equation}
    \begin{aligned}
        &\phantom{{}={}}A_{k+1} (g(\bx_{k+1})-g(\bx^*)) + A_k (g(\bx^*)-g(\bx_k))\\
        & \leq a_kg_k - a_kg(\bx^*) + A_{k+1}g_2(\bx_{k+1}) -A_kg_{2}(\bx_{k}) - a_kg_2(\bz_{k+1}) + \frac{L_g a^2_{k}}{2A_{k+1}}\|\bz_{k+1}-\bz_k\|^2
    \end{aligned}
\end{equation}
By the convexity of $g_2$, $A_{k+1} = A_{k}+a_{k}$, and $\bx_{k+1} = \frac{A_k}{A_k+a_k} \bx_k+\frac{a_k}{A_k+a_k} \bz_{k+1}$ (outlined in Algorithm~\ref{alg:P-AGM-BiO}), i.e. $ A_{k+1}g_2(\bx_{k+1}) -A_kg_{2}(\bx_{k})-a_kg_2(\bz_{k+1}) \leq 0$, the second inequality of this Lemma can be achieved.
\end{proof}
Hence, with the additional Assumption~\ref{ass:additional}, by replicating the analysis outlined in Section~\ref{sec:analysis}, we can derive identical complexity results for Algorithm~\ref{alg:P-AGM-BiO} in either the compact domain setting or with the H\"olderian error bounds on $g$.

\section{Connection with the Polyak Step Size}\label{sec:polyak}

In this section, we would like to highlight the connection between our algorithm's projection step (outlined in Step 6 of Algorithm~\ref{alg:AGM-BiO}) and the Polyak step size. To make this connection, we first without loss of generality, replace $g_{k}$ with $g^{*}$. It is a reasonable argument, as $g_k$ values are close to $g^*$, a point highlighted in \eqref{eq:convergence_g}. In addition, we further assume that the set $\mathcal{Z}=\mathbb{R}^n$ to simplify the expressions. Given these substitutions, the projection step in our AGM-BiO method is equivalent to solving the following problem:
\begin{align*}
   &\min\quad  \|\bx-\bx_{k}\|^{2} \\
   & \hbox{s.t.} \quad g(\bx_k)+\fprod{\nabla g(\bx_k),\bx-\bx_k}\leq g^* 
\end{align*}
In other words,$\bx_{k+1}$ is the unique solution of the above quadratic program with a linear constraint.  By writing the optimality conditions for the above problem and considering $\lambda$ as the Lagrange multipliers associated with the linear constraint, we obtain that 
\begin{equation*}
\begin{cases}
    \bx_{k+1} = \bx_{k} - \lambda \nabla g(\bx_{k}) \\
    \lambda (g(\bx_k)+\fprod{\nabla g(\bx_k),\bx_{k+1}-\bx_k} - g^{*})=0\\
    \lambda \geq0 
\end{cases}
\end{equation*}
Given the fact that $\bx_{k+1}\neq \bx_{k}$, we can conclude that $\lambda \neq 0$, and hence we have 
\begin{equation*}
\begin{cases}
    \bx_{k+1} = \bx_{k} - \lambda \nabla g(\bx_{k}) \\
    g(\bx_k)+\fprod{\nabla g(\bx_k),\bx_{k+1}-\bx_k} -g^{*}=0\\
    \lambda >0 
\end{cases}
\end{equation*}
By replacing $ \bx_{k+1}$ in the second expression with its expression in  the first equation we obtain that 
\begin{equation*}
    \lambda = \frac{g(\bx_{k}) - g^{*}}{\|\nabla g(\bx_{k})\|^{2}}.
\end{equation*}
which is exactly the Polyak step size in the literature \cite{polyak1987introduction}. To solve a bilevel optimization problem, we intend to do gradient descent for both upper- and lower-level functions. Tuning the ratio of upper- and lower-level step size is generally hard. However, by connecting the projection step with the Polyak step size, we observe that the stepsize for the lower-level objective is auto-selected as the Polyak stepsize in our method. In other words, it is one of the advantages of our algorithm that we do not need to choose the lower-level stepsize or ratio of the upper- and lower-level stepsize theoretically or empirically.

\section{Experiment Details}\label{sec:exp_deatils}
In this section, we include more details of the numerical experiments in Section~\ref{sec:experiment}. All simulations are implemented using MATLAB R2022a on a PC running macOS Sonoma with an Apple M1 Pro chip and 16GB Memory.



\subsection{Over-parametrized Regression}
\textbf{Dataset generation.} The original Wikipedia Math Essential dataset \cite{rozemberczki2021pytorch} composes of a data matrix of size 1068 × 731. We randomly select one of the columns as the outcome vector $\bb \in \mathbb{R}^{1068}$ and the rest to be a new matrix $\bA \in \mathbb{R}^{1068\times730}$. We set the constraint parameter $\lambda = 1$ in this experiment, i.e., the constraint set is given by $\mathcal{Z}=\left\{\boldsymbol{\beta} \mid\|\boldsymbol{\beta}\|_{2} \leq 1 \right\}$. 

\textbf{Implementation details.} To be fair, all the algorithms start from the origin as the initial point. For our AGM-BiO method, we set the target tolerances for the absolute suboptimality and infeasibility to $\epsilon_{f} = 10^{-4}$ and $\epsilon_{g} = 10^{-4}$, respectively. We choose the stepsizes as $a_k = 10^{-2}(k+1)/(4L_f)$. In each iteration, we need to do a projection onto an intersection of a $L_2$-ball and a halfspace, which has a closed-form solution. For a-IRG, we set $\eta_{0} = 10^{-3}$ and $\gamma_{0} = 10^{-3}$. For CG-BiO, we obtain an initial point with FW gap of the lower-level problem less than $\epsilon_{g}/2 = 5 \times 10^{-5}$ and choose stepsize $\gamma_{k} = 10^{-2}/(k+2)$. For Bi-SG, we set $\eta_{k} = 10^{-2}/(k+1)^{0.75}$ and $t_{k} = 1/L_{g} = 1/\lambda_{max}(\bA_{tr}^{\top}\bA_{tr}) = 1.5 \times 10^{-4}$. For SEA, we set both the lower- and upper-level stepsizes to be $10^{-4}$. For R-APM, since the lower-level problem does not satisfy the weak sharpness condition, we set $\eta = 1/(K+1) = 1.25\times 10^{-5}$ and $\gamma = 10^{-4} \leq 1/(L_{g}+\eta L_{f}) $. For PB-APG, we set the penalty parameter $\gamma = 10^4$. Note that the lower-level problem in this experiment does not satisfy H\"oderian error bound assumption, so there is no theoretical guarantee for PB-APG.

\subsection{Linear inverse problems}
\textbf{Dataset generation.} We set $\bQ = \bI_{n}$, $\bA = \mathbf{1}_{n}^{\top}$, and $\bb = 1$. The constraint set is selected as $\mathcal{Z} = \mathbb{R}_{+}^{n}$. We choose a low dimensional $(n=3)$ and a high dimensional $(n=100)$ example and run $K = 10^{3}$ number of iterations to compare the numerical performance of these algorithms, respectively.

\textbf{Implementation details.} To be fair, all the algorithms start from the same initial point randomly chosen from $\mathbb{R}_{+}^{n}$. For our AGM-BiO method, we set the stepsizes as $a_k = \gamma(k+1)/(4L_f)$, where $\gamma = 1/(\frac{2L_{g}}{L_{f}}K^{2/3}+2)$ as suggested in Theorem~\ref{thm:upper_lower_hd}. In each iteration, we need to project onto an intersection of a halfspace and $\mathbb{R}_{+}^{n}$. Since halfspaces and $\mathbb{R}_{+}^{n}$ are both convex and closed set, the projection subproblem can be solved by Dykstra's projection algorithm in \cite{bauschke2011fixed}. For a-IRG, we set $\eta_{0} = 10^{-2}$ and $\gamma_{0} = 10^{-2}$. For Bi-SG, we set $\eta_{k} = 1/(k+1)^{0.75}$ and $t_{k} = 1/L_{g}$. For SEA, we set the lower-level stepsize to be $10^{-2}$ and the upper-level stepsize to be $10^{-2}$. For R-APM, since the lower-level problem does not satisfy the weak sharpness condition, we set $\eta = 1/(K+1)$ and $\gamma = 1/(L_{g}+\eta L_{f})$. For PB-APG, we set the penalty parameter $\gamma = 10^4$. For Bisec-BiO, we choose the target tolerances to $\epsilon_f = \epsilon_g =  10^{-4}$. For comparison purposes, we limit the maximum number of gradient evaluations for each APG call to $10^2$. In this experiment, $L_{f} = 1$ and $L_{g} = n$, where $n$ is the number of dimensions.

\end{document}